\documentclass{article}

\usepackage{amssymb}
\usepackage{latexsym}
\usepackage{amsthm}
\usepackage{enumerate}
\usepackage{epsfig}
\usepackage{color}
\usepackage{float}
\usepackage{subfigure}
\usepackage{amsmath}
\usepackage{ifthen}
\usepackage{makeidx}
\usepackage{lscape}
\usepackage{setspace}
\usepackage{amscd}
\usepackage[nobysame]{amsrefs}
\usepackage{xspace}

\newboolean{draft}
\setboolean{draft}{false}

\newcommand{\mylabel}[1]{\ifthenelse{\boolean{draft}}{\label{#1}{\bf [#1]}}{\label{#1}}}

\def\a{\alpha}

\def\C{\mathcal{C}}
\def\g{\gamma}

\def\E{\mathbb{E}}
\def\D{\Delta}
\def\H{\mathcal{H}}

\def\l{\lambda}
\def\N{\mathbb{N}}
\def\O{\Omega}

\def\s{\sigma}
\def\S{\Sigma}

\def\w{\omega}
\def\Z{\mathbb{Z}}

\def\e{\epsilon}
\def\half{\tfrac{1}{2}}

\def\P{\mathbb{P}}
\def\le{\leqslant}
\def\ge{\geqslant}

\def\PMF{\mathcal{PMF}}

\def\Grel{\widehat G}

\def\Bhat{\widehat B}

\def\gz{G^{\Z_+}}
\def\fmin{\mathcal{F}_{min}}

\def\pA{pseudo-Anosov\xspace}

\def\Kstability{K_1}
\def\Kfellowa{K_2}
\def\Kfellowb{K_3}
\def\Kdisjointa{K_4}
\def\Kdisjointb{K_5}
\def\Knest{K_6}
\def\Knesta{K_7}
\def\Kbound{K_8}
\def\Kexp{K_9}
\def\Kmun{K_{10}}

\newcommand{\norm}[1]{|#1|}

\newcommand{\dhat}[1]{\widehat d (#1)}
\newcommand{\nhat}[1]{\dhat{1,#1}}
\newcommand{\dc}[1]{d_\C(#1)}

\newcommand{\mun}[1]{\mu_{#1}}
\newcommand{\rmun}[1]{\widetilde \mu_{#1}}
\newcommand{\rmu}{\widetilde \mu}
\newcommand{\rnu}{\widetilde \nu}
\newcommand{\fix}[1]{\text{fix}(#1)}

\newtheorem{theorem}{Theorem}[section]
\newtheorem{lemma}[theorem]{Lemma}

\newtheorem{proposition}[theorem]{Proposition}

\newtheorem*{theorem:main}{Theorem \ref{theorem:main}}

\theoremstyle{definition}

\newtheorem{claim}[theorem]{Claim}

\newenvironment{definition}{\stepcounter{theorem} \addvspace{.5\baselineskip} \noindent\textbf{Definition \thetheorem. }}{\hfill $\Diamond$\vspace{.5\baselineskip}}

\numberwithin{equation}{section}

\newcounter{case}

\restylefloat{figure}

\begin{document}


\title{Linear progress in the complex of curves}
\author{Joseph Maher\footnote{email: maher@math.okstate.edu}}
\date{\today}

\maketitle

\begin{abstract}
  We show that a random walk on the mapping class group of an
  orientable surface of finite type makes linear progress in the
  relative metric, which is quasi-isometric to the complex of
  curves.

Subject code: 37E30, 20F65, 60J10.

\end{abstract}

\tableofcontents

\section{Introduction}

Let $\S$ be an orientable surface of finite type, which is not a
sphere with three or fewer punctures. The mapping class group $G$ of
$\S$ is the group of orientation preserving diffeomorphisms of $\S$,
modulo those isotopic to the identity. Let $\mu$ be a probability
distribution on $G$.  We may use $\mu$ to generate a random walk on
$G$, which is a Markov chain on $G$ with transition probabilities
$p(x,y) = \mu(x^{-1}y)$, and we will assume we start at the identity
at time zero. The path space for the random walk is the probability
space $(\gz, \P)$, where the product $\gz$ is the collection of all
sample paths, and the measure $\P$ is determined by $\mu$. Let $w_n$
be the random variable corresponding to projection onto the $n$-th
factor. So if $\w$ is a sample path, $w_n(\w)$ is the location of the
path at time $n$, and the distribution of $w_n$ is given by the
$n$-fold convolution of $\mu$ with itself, which we shall write
$\mun{n}$.  We say that the random walk has a linear rate of escape if
the limit $\ell = \lim_{n \to \infty} \tfrac{1}{n} \norm{w_n(\w)}$
exists for almost all sample paths $\w$, and $\ell$ is strictly
greater than zero almost surely.  Here $\norm{g}$ is the length of the
group element $g$ in the word metric on the group $G$.  Kesten
\cites{kesten1,kesten2} and Day \cite{day} showed that an irreducible
random walk on a non-amenable group has a linear rate of escape,
assuming the random walk has finite first moment, i.e. the expected
value of the distance the random walk moves in one step is finite. The
mapping class group contains non-abelian free subgroups, so in
particular is non-amenable.  Therefore a random walk on the mapping
class group makes linear progress in the word metric on the group.

Masur and Minsky \cite{mm1} show that there is a relative metric on
the mapping class group, under which the group is quasi-isometric to
the complex of curves. A relative metric is a word metric on the group
with respect to an infinite generating set, consisting of a finite
generating set, union a finite collection of subgroups. In this case,
the collection of subgroups consists of stabilizers of simple closed
curves $\a_i$, where the $\a_i$ consist of representatives for orbits
of simple closed curves under the action of the mapping class group.
The complex of curves is a simplicial complex whose vertices are
isotopy classes of essential simple closed curves, and whose simplices
are spanned by disjoint collections of simple closed curves. The
quasi-isometry may be explicitly described by choosing a basepoint
$x_0$ for the complex of curves, and sending $g$ to $g(x_0)$. We will
write $\nhat{g}$ for the length of the group element $g$ in the
relative metric on $G$, and this is coarsely equivalent to the
distance $g$ moves the basepoint in the complex of curves. In this
paper we show that a random walk in the mapping class group makes
linear progress in the relative metric.

\begin{theorem}
\mylabel{theorem:main}
Let $G$ be the mapping class group of an orientable surface of finite
type, which is not a sphere with three or fewer punctures, and
consider the random walk generated by a probability distribution $\mu$
whose support is bounded in the relative metric, which generates a
non-elementary subgroup of the mapping class group and which has
finite first moment. Then there is a constant $\ell > 0$ such that $
\lim_{n \to \infty} \tfrac{1}{n}\nhat{w_n} = \ell $ almost surely.
\end{theorem}

The relative metric is an improper metric on the mapping class group,
and therefore not quasi-isometric to the word metric, so linear
progress in the word metric does not immediately imply linear progress
in the relative metric. We will begin with a brief discussion of the
mapping class group of the torus, as in this case it is easy to see
that the result holds.

The mapping class group of the torus is isomorphic to the group
$SL(2,\Z)$. As there are no non-parallel disjoint essential simple
closed curves on the torus, the curve complex for the torus is usually
defined to have edges connecting pairs of curves that intersect
exactly once. The curve complex for the torus is the Farey
triangulation of the unit disc, and $SL(2,\Z)$ with its word metric is
quasi-isometric to the dual graph, which is a trivalent tree. This is
illustrated below in Figure \ref{picture18}.

\begin{figure}[H]
\begin{center}
\epsfig{file=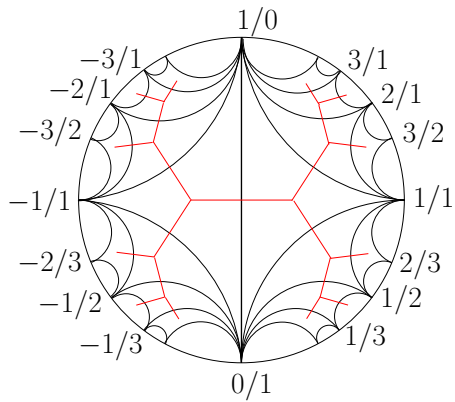, height=150pt}
\end{center}
\caption{The Farey triangulation.} \mylabel{picture18}
\end{figure}

For simplicity, consider a non-backtracking random walk on the
trivalent tree. Such a path may be described by a sequence
LRLLRL\ldots, where an L denotes a left turn at a vertex, and an R
denotes a right turn at a vertex. Such a path makes uniform progress
in the trivalent graph. The path travels distance one in the relative
space whenever the next letter in the sequence is different from the
previous one, and this occurs with probability one-half. So the random
walk makes progress in the relative space on average at half the rate
it makes progress in the trivalent tree.

We now indicate the argument we use for more complicated mapping class
groups.  Imagine starting at a basepoint $x_0$ in hyperbolic space,
and travelling some large distance $R$. The set of points you may
arrive at is given by a sphere centered at $x_0$. Now travel another
large distance $S$, as illustrated below in Figure \ref{picture19}.

\begin{figure}[H]
\begin{center}
\epsfig{file=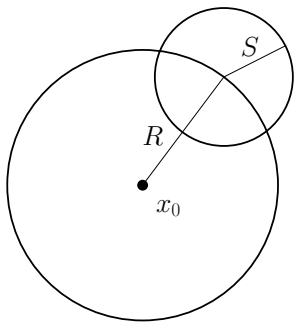, height=120pt}
\end{center}
\caption{Spheres in hyperbolic space.} \mylabel{picture19}
\end{figure}

Due to negative curvature, nearly all the volume of the second sphere
of radius $S$ lies outside the sphere of radius $R$. We wish to
translate this intuition into our setting. Our space is the mapping
class group $G$ with a relative metric, which we shall denote $\Grel$,
which is quasi-isometric to the complex of curves. Masur and Minsky
\cite{mm1} showed that this space is a (non-proper)
$\delta$-hyperbolic space, and Klarreich \cite{klarreich} identified
the Gromov boundary of the space as the space of foliations in $\PMF$
which contain no closed trajectories.  Let $\mu$ be a probability
distribution on $G$, whose support generates a non-elementary
subgroup. Kaimanovich and Masur showed that for the random walk
determined by $\mu$, almost all sample paths converge to uniquely
ergodic, and hence minimal, foliations and this defines a harmonic
measure $\nu$ on the boundary, where $\nu(X)$ is the probability that
a sample path converges to a foliation contained in the set $X$. The
harmonic measure $\nu$ governs the long time behaviour of the sample
paths, and is the weak-$\star$ limit of the $n$-fold convolutions of
$\mu$ on $\Grel$ union its boundary. In particular, if we start at the
identity and consider all sample paths of length $R$, for $R$ large,
the distribution of endpoints looks similar to $\nu$, at least when
viewed from the identity. If we now continue the random walk for
another $S$ steps, negative curvature leads us to expect that most of
the new sample points will lie a definite distance further away from
the origin.

In order to make this intuition precise, we need some way of comparing
the harmonic measure $\nu$ with the convolution measures $\mun{n}$. We
will compare the measures on sets which are \emph{halfspaces}. A
halfspace $H(x,y)$ consists of all points closer (in $\Grel$) to $y$
than $x$, and we will be most interested in halfspaces $H(1,x)$, where
$1$ is the identity element in $G$. In Section \ref{section:decay} we
will show that the harmonic measure of a halfspace $H(1,x)$ decays
exponentially in $\nhat{x}$, i.e. there is a constant $L < 1$ such
that $\nu(\overline{H(1,x)}) \le L^{\nhat{x}}$, for all $x$
sufficiently far from $1$, where $\overline{H(1,x)}$ is the closure of
$H(1,x)$. Furthermore, we will show there is a constant $Q$ such that
$\mun{n}(H(1,x)) \le QL^{\nhat{x}}$, for all $x$ sufficiently far from
$1$, and these estimates will allow us to relate the harmonic and
convolution measures. We briefly indicate why one expect these
estimates to hold. First observe that there is some number $K$ such
that the harmonic measure of all halfspaces $H(1,x)$, with $\nhat{x}
\ge K$ is bounded away from $1$, at most $1-\e$, say.  Then given a
halfspace $H(1,x_1)$, where $\nhat{x_1}$ is large, one may construct a
nested sequence of half spaces $H(1,x_n) \supset H(1,x_{n-1}) \supset
\ldots \supset H(1,x_1)$. The number of such halfspaces is linear in
$\nhat{x_i}$, and furthermore, we may assume that the distance between
any point in $H(1,x_i)$ and any point in the complement of
$H(1,x_{i+1})$ is larger than $K$. The conditional probability that a
sample path converges into $H(1,x_i)$, given that it hits a point in
the complement of $H(1,x_{i+1})$ is at most $1-\e$, so the harmonic
measure of the innermost halfspace is at most $(1-\e)^n$. In order to
bound $\mun{n}$ in terms of $\nu$, suppose a large amount of the mass
of $\mun{n}$ is contained in $H(1,x_i)$, then at most $(1-\e)
\mun{n}(H(1,x_i))$ of the mass can escape back out in to the
complement of $H(1,x_{i+1})$, so this gives the upper bound for
$\mun{n}$.

We will use the estimates described above to find a positive lower
bound on the expected extra distance from $1$ obtained by taking an
extra $m$ steps, after a random walk of length $n$, for $m$
sufficiently large, i.e. we will show $\E(\nhat{w_{n+m}} -
\nhat{w_{n}}) \ge \delta > 0$. Therefore $\E(\nhat{w_{km}}) \ge
k\delta$, so the expected distance from the identity grows linearly,
and a standard application of Kingman's subadditive ergodic theorem
\cite{kingman} shows that this will then be true for almost all sample
paths as well.

We remark that there are distance non-increasing maps to the relative
space from other useful spaces on which the mapping class group acts,
such as Teichm\"uller space and the pants complex, so linear progress
in the complex of curves implies linear progress in these spaces too.
In particular, work of Duchin \cite{duchin}, shows that linear progress
in Teichm\"uller space implies that for almost all sample paths there
is a geodesic which the random walk tracks sublinearly, at least for
the parts of the geodesic in the thick part of Teichm\"uller space.

In Section \ref{section:preliminaries}, we recall some standard
definitions and set up some notation. In Section
\ref{section:halfspaces} we prove some useful results about halfspaces
in non-proper $\delta$-hyperbolic metric spaces. As we do not assume
that $\mu$ is symmetric, it will be convenient for us to know that a
semi-group in the mapping class group contains a pair of independent
\pA elements, if and only if it generates a non-elementary subgroup,
and we show this in Section \ref{section:semigroups}, in a straight
forward extension of some results of Ivanov \cite{ivanov}. In Section
\ref{section:decay} we show that the harmonic measure of halfspaces in
the relative space decays exponentially in the distance of the
halfspace from the basepoint, and obtain the estimate for the
convolution measures. Finally in Section \ref{section:linear progress}
we find a positive lower bound for the expected difference between
$\nhat{w_{n+m}}$ and $\nhat{w_n}$, and then apply Kingman's subadditive
ergodic theorem to show that a random walk makes linear progress in
the relative space.

\subsection{Acknowledgements}

I would like to thank the referee for many helpful comments.  This
work was partially supported by NSF grant DMS-0706764.

\section{Preliminaries} \label{section:preliminaries}

Let $\S$ be an orientable surface of finite type, i.e. a surface of
genus $g$ with $p$ marked points, usually referred to as punctures.
The mapping class group $G$ of $\S$ consists of orientation preserving
diffeomorphisms of $\S$ which preserve the punctures, modulo those isotopic to
the identity. For the purposes of this paper we shall assume that $\S$
is not a sphere with three or fewer punctures.

The collection of essential simple closed curves in the surface may be
made in to a simplicial complex, called the the \emph{complex of
  curves}, which we shall denote $\C(\S)$.  The vertices of this
complex are isotopy classes of simple closed curves in $\S$, and a
collection of vertices spans a simplex if representatives of the
curves can be realised disjointly in the surface.  The complex of
curves is a finite dimensional simplicial complex, but it is not
locally finite. We will write $\C_0(\S)$ to denote the vertices of the
simplicial complex $\C(\S)$, which is the set of isotopy classes of
simple closed curves.  We will write $\dc{x,y}$ for the distance in
the one-skeleton between two vertices $x$ and $y$ of the complex of
curves. We will always consider the complex of curves to have a
basepoint $x_0$, which we can take to be one of the curves
corresponding to a standard generating set for the mapping class
group. The mapping class group acts by simplicial isometries on the
complex of curves. For certain sporadic surfaces the definition above
produces a collection of disconnected points, and so a slightly
different definition is used. If the surface is a torus with at most
one puncture, then two vertices are connected by an edge if the
corresponding simple closed curves may be isotoped to intersect
transversely exactly once. If the surfaces is a four punctured sphere,
then two vertices are connected by an edge if the corresponding simple
closed curves may be isotoped to intersect transversely in two points.
In both of these cases, the resulting curve complex is isomorphic to
the Farey graph.

A geodesic metric space is \emph{$\delta$-hyperbolic} if every
geodesic triangle is $\delta$-slim, i.e. each edge is contained in a
$\delta$-neighbourhood of the other two.  Masur and Minsky \cite{mm1}
have shown that the complex of curves is $\delta$-hyperbolic.

The mapping class group is finitely generated, so any choice of
generating set $A$ gives rise to a word metric on $G$, and any two
different choices of finite generating set give quasi-isometric word
metrics.  Given a group $G$, and a collection of subgroups $\H = \{
H_i \}_{i \in I }$, we define the \emph{relative length} of a group
element $g$ to be the length of the shortest word in the typically
infinite generating set $A \cup \H$.  This defines a metric on $G$
called the \emph{relative metric}, which depends on the choice of
subgroups $\H$.  We will write $\Grel$ to denote the group $G$ with
the relative metric, which we shall also refer to as the
\emph{relative space}.  We say a finitely generated group $G$ is
\emph{weakly relatively hyperbolic}, relative to a finite list of
subgroups $\H$, if the relative space $\Grel$ is $\delta$-hyperbolic.

We may consider the relative metric on the mapping class group with
respect to the following collection of subgroups.  Let $\{\alpha_1,
\ldots, \alpha_n\}$ be a list of representatives of orbits of simple
closed curves in $\S$, under the action of the mapping class group.
Let $H_i = \fix{\alpha_i}$ be the subgroup of $G$ fixing $\alpha_i$.
Masur and Minsky \cite{mm1} have shown that the resulting relative
space is quasi-isometric to the complex of curves. As the complex of
curves is $\delta$-hyperbolic, this shows that the mapping class group
is weakly relatively hyperbolic. Klarreich \cite{klarreich}, see also
Hamenst\"adt \cite{hamenstadt}, showed that the Gromov boundary of the
complex of curves is the space $\fmin$, which we now describe. The
space $\fmin$ consists of the subset of $\PMF$ consisting of
foliations which contain no closed trajectories, up to equivalence,
where two measured foliations are equivalent if they are topologically
equivalent. The topology on $\fmin$ is the induced topology from
$\PMF$.

We now review some background on random walks on groups, see for
example Woess \cite{woess}.  Let $G$ be the mapping class group of an
orientable surface of finite type, which is not a sphere with three or
fewer punctures, and let $\mu$ be a probability distribution on $G$.
We may use the probability distribution $\mu$ to generate a Markov
chain, or \emph{random walk} on $G$, with transition probabilities
$p(x,y) = \mu(x^{-1}y)$. We shall always assume that we start at time
zero at the identity element of the group.  The \emph{path space} for
the random walk is the probability space $(G^{\Z_+},\P)$, where
$G^{\Z_+}$ is the set of all infinite sequences of elements $G$. We
will write $w_n$ for the random variable corresponding to projection
onto the $n$-th factor, which gives the position of the sample path at
time $n$. The position of the random walk at time $n$ may be described
as the product $m_1 m_2 \dots m_n$, where the $m_i$ are the
\emph{increments} of the random walk, i.e. the $m_i$ are a sequence of
independent $\mu$-distributed random variables. Therefore the
distribution of random walks at time $n$ is given by the $n$-fold
convolution of $\mu$, which we shall write as $\mun{n}$, and we shall
write $p^{(n)}(x,y)$ for the probability that you go from $x$ to $y$
in $n$ steps. The probability measure $\P$ is determined by $\mun{n}$
using the Kolmogorov extension theorem. The \emph{reflected} random
walk is the walk generated by the reflected measure $\rmu$, where
$\rmu(g) = \mu(g^{-1})$.  The \emph{Bernoulli shift} in the space of
increments of the random walk determines a measure-preserving ergodic
transformation on $(\gz, \P)$ determined by $(Uw)_n =
w_1^{-1}w_{n+1}$.

We shall always require that the group generated by the support of
$\mu$ is \emph{non-elementary}, which means that it contains a pair of
\pA elements with distinct fixed points in $\PMF$, the space of
projective measured foliations on the surface. We do not assume that
the probability distribution $\mu$ is symmetric, so the group
generated by the support of $\mu$ may be strictly larger than the
semi-group generated by the support of $\mu$. Finally, we shall always
assume that the probability distribution $\mu$ has finite first moment
with respect to the word metric on $G$.

In \cite{maher}, we showed that it followed from results of
Kaimanovich and Masur \cite{km} and Klarreich \cite{klarreich}, that a
sample path converges almost surely to a uniquely ergodic foliation in
the Gromov boundary of the relative space. This gives a measure $\nu$
on $\fmin$, known as harmonic measure. The harmonic measure $\nu$ is
$\mu$-stationary, i.e. \[ \nu(X) = \sum_{g \in G} \mu(g)\nu(g^{-1}X). \]

\begin{theorem} \cites{km, klarreich, maher}
Consider a random walk on the mapping class group of an orientable
surface of finite type, which is not a sphere with three or fewer
punctures, determined by a probability distribution $\mu$ such that
the group generated by the support of $\mu$ is non-elementary. Then
a sample path $\{ w_n(\w) \}$ converges to a uniquely ergodic
foliation in the Gromov boundary $\fmin$ of the relative space $\Grel$
almost surely, and the distribution of limit points on the boundary
is given by a unique $\mu$-stationary, non-atomic measure $\nu$ on $\fmin$.
\end{theorem}

We remark that the measure $\nu$ is supported on the uniquely ergodic
foliations, which are a subset of $\PMF$, so we may think of $\nu$ as
a measure on $\PMF$, with zero weight on all the non-uniquely ergodic
measures.

\section{Halfspaces} \label{section:halfspaces}

In this section we give detailed proofs of various useful properties
of halfspaces in a non-locally compact $\delta$-hyperbolic space.
These properties are presumably well known for locally compact
$\delta$-hyperbolic spaces, but we provide complete proofs to verify
that these properties hold in the non-locally compact case. For
consistency with the other sections of this paper, we will denote our
non-locally compact $\delta$-hyperbolic metric space by $\Grel$, and
we will write $\dhat{x,y}$ for the distance in $\Grel$ between two
points $x$ and $y$. Furthermore, we will always have a distinguished
base point, which we shall call $1$.

Two points $a$ and $b$ in a metric space define a \emph{halfspace}
$H(a,b)$ consisting of all those points which are at least as close to
$b$ as to $a$, i.e. $H(a,b) = \{ x \in \Grel \mid \dhat{x,b} \leqslant
\dhat{x,a} \}$.  The main two results of this section are Propositions
\ref{prop:disjoint} and \ref{prop:nested}, which we now briefly
describe. Proposition \ref{prop:disjoint} says if two halfspaces
$H(1,x)$ and $H(1,y)$ are small and far apart, then any other small
halfspace hits at most one of $H(1,x)$ or $H(1,y)$. The halfspaces
$H(1,x)$ and $H(1,y)$ are small if $\nhat{x}$ and $\nhat{y}$ are
large, and they are far apart if the geodesic $[x,y]$ passes close to
the basepoint $1$. In this case any other halfspace $H(1,z)$ hits at
most one of $H(1,x)$ or $H(1,y)$, as long as $\nhat{z}$ is
sufficiently large.  Proposition \ref{prop:nested} says that given a
halfspace $H(1,x)$, we may choose a point $y$ on a geodesic $[1,x]$
close to $x$ such that the halfspace $H(1,x)$ is contained in the
halfspace $H(1,y)$. We also obtain specific bounds on how large the
halfspace $H(1,x)$ appears when viewed from any point in $H(y,1)$, and
similar bounds on how large the halfspace $H(y,1)$ appears when viewed
from any point in $H(1,x)$.  To be more precise, we show that for any
point $a$ in $H(y,1)$, the halfspace $H(1,x)$ is contained in a
halfspace $H(a,b)$, with an explicit lower bound on $\dhat{a,b}$.
Furthermore, for any point $b$ in $H(1,x)$, there is a halfspace
$H(b,a)$ such that $H(y,1) \subset H(b,a)$, again, with an explicit
lower bound on $\dhat{a,b}$.

We begin with some elementary observations about nearest point
projections.  In a $\delta$-hyperbolic space nearest point projections
onto quasi-convex sets are coarsely well defined.  We now show that
for any set $X$, the $K$-neighbourhoods of $X$, and a particular
choice of shortest path to $X$, have bounded intersection.

\begin{proposition} \mylabel{prop:bounded}
Let $X$ be a set, and let $[z,p]$ be a minimal length geodesic from a
point $z$
to $X$. Then the intersection of the $K$-neighbourhoods of
$X$ and $[z,p]$ is contained in a $3K$-neighbourhood of $p$.
\end{proposition}

\begin{figure}[H]
\begin{center}
\epsfig{file=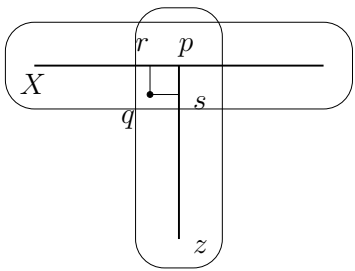, height=100pt}
\end{center}
\caption{A shortest path from $z$ to $X$.} \mylabel{picture8}
\end{figure}

\begin{proof}
Choose a point $q$ in the intersection of $N_{K}(X)$ and
$N_{K}([z,p])$. There are points $r \in X$ and $s \in [z,p]$ such that
$\dhat{r,q} \le K$ and $\dhat{s,q} \le K$, as illustrated above in
Figure \ref{picture8}. If $\dhat{p,q} > 3K$, then as $\dhat{s,q} \le
K$, this implies that $\dhat{p,s} > 2K$. Therefore $r$ is a closer
point on $[x,y]$ to $s$, and hence to $z$, than $p$, which contradicts
the fact that $p$ is a closest point on $X$ to $z$. So $\dhat{p,q}
\leqslant 3K$, and so $N_{K}(X) \cap N_{K}([z,p]) \subset N_{3K}(p)$,
as required.
\end{proof}

This shows that, up to additive error, the shortest way to get from a
point $z$ to a point $x$ on a geodesic, is to head to the closest
point to $z$ on the geodesic, and then run along the geodesic to $x$.
We will make extensive use of this fact, so we record it here as a
proposition.

\begin{proposition} \mylabel{prop:projection} 
Let $[x,y]$ be a geodesic from $x$ to $y$, and $p$ a closest point on
$[x,y]$ to $z$.  Any geodesic $[x,z]$ from $x$ to $z$ intersects a
$3\delta$-neighbourhood of $p$, so $[x,p] \cup [p,z]$ is contained in
a $3\delta$-neighbourhood of $[x,z]$.  In particular $\dhat{x,p} +
\dhat{p,z} -6\delta \leqslant \dhat{x,z} \leqslant \dhat{x,p} +
\dhat{p,z}$.
\end{proposition}

\begin{proof}
The right hand inequality is just the triangle inequality. We now
justify the left hand inequality. By thin triangles, any geodesic
$[x,z]$ from $x$ to $z$ is contained in a $\delta$-neighbourhood of
the union of geodesics $[x,p]$ and $[p,z]$. If $[p,x]$ and $[p,z]$ had
long initial segments that fellow travelled, then the distance from
$x$ to $z$ might be much shorter than the sum of the distances from
$x$ to $p$ and from $z$ to $p$. However, this would contradict the
fact that $p$ was a closest point on $[x,y]$ to $z$. 
To be precise, Proposition \ref{prop:bounded} implies that the
intersection of the $\delta$-neighbourhoods of $[x,p]$ and $[p,z]$ is
contained in a $3\delta$-neighbourhood of $p$.  This means that any
geodesic from $z$ to $x$ must pass within $3\delta$ of $p$, so
$\dhat{z,p} + \dhat{p,x} - 6\delta \le \dhat{z,x}$, as required.
\end{proof}

An immediate consequence of Proposition \ref{prop:projection} is the
following quantitative version of the fact that nearest point
projection is coarsely well defined.

\begin{proposition} \mylabel{prop:npp}
Let $p$ and $q$ be nearest points to $z$ on a geodesic $[x,y]$. Then
$\dhat{p,q} \le 6\delta$.
\end{proposition}

Let $a$ and $b$ have nearest point projections $p$ and $q$ onto a
geodesic $[x,y]$. We now show that if $\dhat{p,q} > 14\delta$ apart,
then, up to additive error, the geodesic from $a$ to $b$ goes from $a$
to $p$, then runs along the geodesic from $p$ to $q$, and then heads
back out to $b$.

\begin{proposition} \mylabel{prop:double}
Let $[x,y]$ be a geodesic and let $p$ be a closest point on $[x,y]$ to
$a$, and let $q$ be a closest point on $[x,y]$ to $b$. If $\dhat{p,q}
> 14\delta$ then $\dhat{a,b} \geqslant \dhat{a,p} + \dhat{p,q} +
\dhat{q,b} - 24\delta$.
\end{proposition}

\begin{proof}
Let $[x,y]$ be a geodesic from $x$ to $y$. Let $[a,p]$ be a minimal
length geodesic from $a$ to $[x,y]$, and let $[b,q]$ be a minimal
length geodesic from $b$ to $[x,y]$.
First we show that $N_{2\delta}([a,p])$ and $N_{2\delta}([b,q])$ are
disjoint. Suppose not, then let $r$ be a point in $N_{2\delta}([a,p])
\cap N_{2\delta}([b,q])$. Then there are points $s \in [a,p]$ and $t
\in [b,q]$ such that $\dhat{s,r} \le 2\delta$ and $\dhat{t,r} \le
2\delta$, as illustrated in Figure \ref{picture10} below.

\begin{figure}[H]
\begin{center}
\epsfig{file=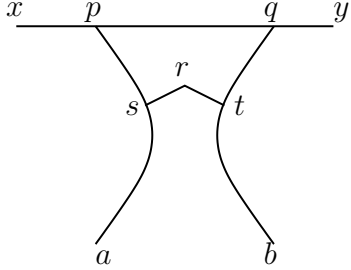, height=100pt}
\end{center}
\caption{Shortest paths from $a$ and $b$ to $[x,y]$.} \mylabel{picture10}
\end{figure}

By the triangle inequality, $\dhat{t,p} \le \dhat{t,s} + \dhat{s,p}$.
As $\dhat{t,s} \le 4\delta$, this shows that
\begin{align}
\dhat{t,p} & \leqslant 4\delta + \dhat{s,p}. \label{equation:1}
\intertext{By Proposition \ref{prop:projection}, the path from $t$ to
  $p$ via $q$ is almost a geodesic, i.e.}
\dhat{t,q} + \dhat{q,p} - 6\delta & \le \dhat{t,p}. \label{equation:2}
\intertext{Combining \eqref{equation:1} and \eqref{equation:2}, we obtain}
\dhat{t,q} + \dhat{q,p} - 6\delta &\le \dhat{s,p} + 4\delta. \label{equation:3}
\intertext{As $[a,p]$ is a minimal length geodesic from $a$ to $[x,y]$, it is
also a minimal length geodesic from any point on $[a,p]$ to $[x,y]$.
Therefore $p$ is a closest point on $[x,y]$ to $s$, so the distance
from $s$ to $p$ is less than or equal to the distance from $s$ to
$[x,y]$ by a path through $t$. This implies that}
\dhat{s,p} & \leqslant 4\delta + \dhat{t,q}. \label{equation:4}
\end{align}
Combining \eqref{equation:3} and \eqref{equation:4}, and subtracting
$\dhat{t,q}$ from both sides, implies that $\dhat{p,q} \leqslant
14\delta$.  However we assumed that $\dhat{p,q} > 14\delta$, so this
implies that $N_{2\delta}([a,p])$ and $N_{2\delta}([b,q])$ are in fact
disjoint.

By thin triangles, any geodesic $[a,b]$ is contained in
$2\delta$-neighbourhood of $[a,p] \cup [p,q] \cup [q,b]$. By
Proposition \ref{prop:bounded}, the intersection of
$N_{2\delta}([a,p])$ with $N_{2\delta}([x,y])$ is contained in a
$6\delta$-neighbourhood of $p$. Similarly, the intersection of
$N_{2\delta}([b,q])$ with $N_{2\delta}([x,y])$ is contained in a
$6\delta$-neighbourhood of $q$. As the remaining parts of the
$2\delta$-neighbourhoods outside $N_{6\delta}(p)$ and $N_{6\delta}(q)$
are disjoint, this means that $\dhat{a,b} \ge \dhat{a,p} + \dhat{p,q}
+ \dhat{q,b} - 24\delta$, as required.
\end{proof}

We now show that nearest point projection to a connected subgeodesic
of a geodesic is coarsely equivalent to nearest point projection to the
original geodesic, followed by nearest point projection to the
subgeodesic.

\begin{proposition} \mylabel{prop:stability}
Let $\rho_1$ be nearest point projection onto a geodesic $[a,b]$, and
let $\rho_2$ be nearest point projection onto a subgeodesic $[c,d]
\subset [a,b]$. Then there is a constant $\Kstability$, which only
depends on $\delta$, such that $\dhat{\rho_2(x), \rho_2(\rho_1(x))}
\le \Kstability$, for any point $x$.
\end{proposition}

\begin{proof}
Let $p$ be the nearest point projection of $x$ to $[a,b]$, i.e. $p =
\rho_1(x)$.  If $q$ is another point on $[a,b]$, then the path from
$x$ to $q$ via $p$ is almost a geodesic, Proposition
\ref{prop:projection}, so $\dhat{x,p} + \dhat{p,q} - 6\delta \le
\dhat{x,q} \le \dhat{x,p} + \dhat{p,q}$. Therefore, if $p$ lies in the
subgeodesic $[c,d]$, then $\pi_2(p) = p$, and any point in $[c,d]$
further than $6\delta$ from $p$ is further away from $x$ than $p$. On
the other hand, if $p$ lies outside $[c,d]$, then $\rho_2(p)$ is equal
to one of the endpoints, which we may assume is $c$, up to
relabelling. Again, Proposition \ref{prop:double} implies that any
point on $[c,d]$ more than $6\delta$ away from $c$ is further away
from $x$ than $c$. So we may choose $\Kstability$ to be $7\delta$.
\end{proof}

We now show that that if a path $\g$ lies in a bounded neighbourhood
of a geodesic, then the nearest point projections of any point to the
path and to the geodesic are a bounded distance apart.

\begin{proposition} \mylabel{prop:close}
Let $[x,y]$ be a geodesic, and let $\g$ be any path from $x$ to $y$
contained in a $K$-neighbourhood of $[x,y]$. For any point $z$, let $p$ be a closest
point on $[x,y]$ to $z$, and let $q$ be a closest point on $\g$ to
$z$. Then $\dhat{p,q} \le 3K + 6\delta$.
\end{proposition}

\begin{proof}
Let $p$ be the closest point to $z$ on $[x,y]$. As $\g$ is a path from
$[x,y]$ contained in a $K$-neighbourhood of $[x,y]$, there is a point
$p'$ in $\g$ such that $\dhat{p,p'} \le K$, and hence $\dhat{p',z} \le
\dhat{z,p} + K$. Let $q$ be the closest point on $\g$ to
$z$, and let $q'$ be the closest point on $[x,y]$ to $q$, so the
distance from $\dhat{q,q'} \le K$. This is illustrated below in
Figure \ref{picture16}.

\begin{figure}[H]
\begin{center}
\epsfig{file=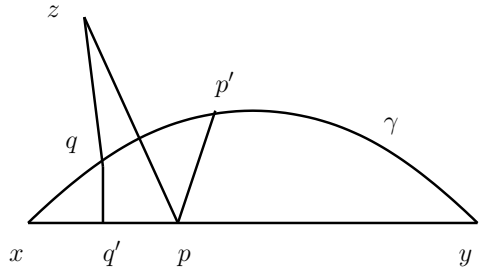, height=100pt}
\end{center}
\caption{A path close to a geodesic $[x,y]$.} \mylabel{picture16}
\end{figure}

By Proposition \ref{prop:projection}, the path from $z$ to $q'$ via $p$
is almost a geodesic, so $\dhat{z,q'} \geqslant \dhat{z,p} + \dhat{p,q'}
- 6\delta$, which implies that $\dhat{z,q} \geqslant \dhat{z,p} +
\dhat{p,q'} - 6\delta - K$. Therefore if $\dhat{p,q'} > 2K + 6\delta$,
then $\dhat{z,q} \ge \dhat{z,p'}$, which contradicts our choice of
$q$ as a closest point on $\g$ to $z$. Therefore $\dhat{p,q} \le 3K
+ 6\delta$, as required.
\end{proof}

Recall that two points $x$ and $y$ define a halfspace $H(x,y)$. We
now show that the image of the halfspace under the nearest point
projection to a geodesic $[x,y]$ between $x$ and $y$ is contained in a
bounded neighbourhood of the half-segment of $[x,y]$ closest to $y$.
As a partial converse, we show that if the nearest point projection of
$z$ lies sufficiently close to $y$, then $z \in H(x,y)$.

\begin{proposition} \mylabel{prop:half}
Let $z \in H(x,y)$, and let $p$ be the nearest point to $z$ on a geodesic
$[x,y]$. Then $\dhat{y,p} \leqslant \tfrac{1}{2}\dhat{x,y} +
3\delta$. Conversely, if $\dhat{y,p} \leqslant \tfrac{1}{2}\dhat{x,y}
-3\delta$, then $z \in H(x,y)$.
\end{proposition}

\begin{proof}
Let $z$ be a point in the halfspace $H(x,y)$, let $[x,y]$ be a geodesic from $x$ to $y$,
and let $p$ be a closest point on $[x,y]$ to $z$. This is illustrated
below in Figure \ref{picture17}.

\begin{figure}[H]
\begin{center}
\epsfig{file=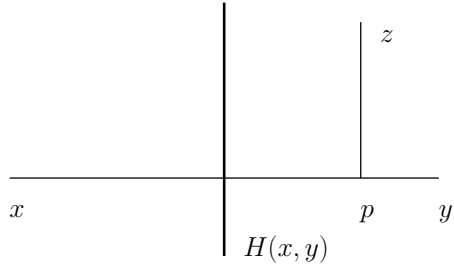, height=100pt}
\end{center}
\caption{A point $z$ in the halfspace $H(x,y)$.} \mylabel{picture17}
\end{figure}

The point $z$ lies in the halfspace $H(x,y)$, so $\dhat{z,y} \leqslant
\dhat{z,x}$.  By the triangle inequality, the distance from $z$ to $x$
is less than or equal to the distance from $z$ to $x$ via $p$, which implies
\begin{align*}
\dhat{z,y} & \leqslant \dhat{z,p} + \dhat{p,x}.
\intertext{By Proposition \ref{prop:projection}, the path from $z$ to $y$ via $p$
is close to being a geodesic, i.e. $\dhat{z,p} + \dhat{p,y} - 6\delta
\leqslant \dhat{z,y}$.}
\dhat{z,p} + \dhat{p,y} - 6\delta & \leqslant \dhat{z,p} + \dhat{p,x}
\end{align*}
We may subtract $\dhat{z,p}$ from both sides. As $p$ lies on the
geodesic $[x,y]$, the distance from $x$ to $y$ is equal to the
distance from $x$ to $p$ plus the distance from $p$ to $y$, i.e.
$\dhat{x,y} = \dhat{x,p} + \dhat{p,y}$. This gives the required
inequality, $\dhat{p,y} \leqslant \tfrac{1}{2}\dhat{x,y}+ 3\delta$.

For the converse, note that if $z$ does not lie in the halfspace
$H(x,y)$, then $z$ lies in the halfspace $H(y,x)$, so $\dhat{p,x}
\leqslant \tfrac{1}{2}\dhat{x,y} +3\delta$. The point $p$ lies on the
geodesic $[x,y]$, so $\dhat{p,x} = \dhat{x,y} - \dhat{p,y}$, which
implies $\tfrac{1}{2}\dhat{x,y} \leqslant \dhat{p,y} + 3\delta$.
Therefore, if $\dhat{y,p} \leqslant \tfrac{1}{2}\dhat{x,y} -3\delta$
then this implies $z \in H(x,y)$, as required.
\end{proof}

We have shown that the image of the nearest point projection of a
halfspace $H(x,y)$ to $[x,y]$ is close to being half of this geodesic
segment. If a geodesic segment $[x,z]$ fellow travels with $[x,y]$ for a
sufficiently large initial segment, then we can estimate the image of
the nearest point projection of $H(x,y)$ onto $[x,z]$. This will be
the case as long as the nearest point projection of $z$ to $[x,y]$ is
sufficiently far from $x$. We now make this precise in the following
proposition.

\begin{proposition} \mylabel{prop:fellow} 
There are constants $\Kfellowa$ and $\Kfellowb$, which only depend on
$\delta$, such that if $z$ has nearest point projection $p$ to a
geodesic $[x,y]$, and $\dhat{p,x} \ge \tfrac{1}{2}\dhat{x,y} +
\Kfellowa$, then the nearest point projection of $H(y,x)$ to $[x,z]$
is distance at most $\tfrac{1}{2}\dhat{x,y} + \Kfellowb$ from $x$.
\end{proposition}

\begin{proof}
We shall choose $\Kfellowa$ to be $27\delta$ and $\Kfellowb$ to be
$18\delta$. Let $H(y,x)$ be the halfspace defined by the pair of
points $x$ and $y$, and let $[x,y]$ be a geodesic from $x$ to $y$. Let
$p$ be a nearest point to $z$ on $[x,y]$.  Let $a \in H(y,x)$, and let
$q$ be the closest point on $[x,y]$ to $a$. This is illustrated in
Figure \ref{picture14} below.

\begin{figure}[H]
\begin{center}
\epsfig{file=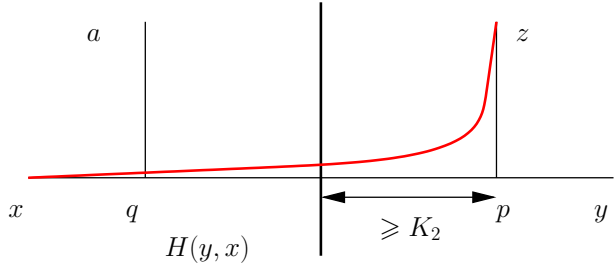, height=100pt}
\end{center}
\caption{Closest point projection of $H(y,x)$ onto $[x,z]$.} \mylabel{picture14}
\end{figure}

By Proposition \ref{prop:half}, nearest point projection maps the
halfspace $H(y,x)$ to roughly half the geodesic $[x,y]$. This implies
that $\dhat{q,x} \le\tfrac{1}{2}\dhat{x,y} + 3\delta$, so $\dhat{q,p}
\ge \Kfellowa - 3\delta$. Let $t$ be a point on $[p,z]$, then as
$\Kfellowa \ge 17\delta$, Proposition \ref{prop:double} implies that
$\dhat{a,t} \ge \dhat{a,q} + \dhat{q,p} + \dhat{p,t} - 24\delta$. As
we chose $\Kfellowa$ to be $27\delta$, this implies $\dhat{p,q} \ge
24\delta$, and so the closest point to $a$ on $[x,p] \cup [p,z]$ lies
on $[x,p]$. As $[x,p] \subset [x,y]$, Proposition \ref{prop:half}
implies that the nearest point projection of $H(y,x)$ to $[x,p] \cup
[p,z]$ is distance at most $\tfrac{1}{2}\dhat{x,y} + 3 \delta$ from
$x$.

Let $r$ be the nearest point to $a$ on the geodesic $[x,z]$, and let
$s$ be the nearest point to $a$ on the path $[x,p] \cup [p,z]$. By
Proposition \ref{prop:projection}, the path from $x$ to $z$ via $p$ is
contained in a $3\delta$-neighbourhood of a geodesic $[x,z]$, so as
nearest point projections to close paths are close, Proposition
\ref{prop:close}, this implies $\dhat{r,s} \le 3(3\delta) + 6\delta =
15\delta$.  Therefore the nearest point projection of $H(y,x)$ to
$[x,z]$ is distance at most $\tfrac{1}{2}\dhat{x,y} + 18 \delta$ from
$x$, and so we may choose $\Kfellowb$ to be $18\delta$.
\end{proof}

We now show that if $H(1,x)$ and $H(1,y)$ are two halfspaces far from
the origin, such that the geodesic from $[x,y]$ passes close to the
origin, then any other halfspace far from the origin hits at most one
of $H(1,x)$ or $H(1,y)$.

\begin{proposition} \mylabel{prop:disjoint}
There are constants $\Kdisjointa$ and $\Kdisjointb$, which only depend on $\delta$,
such that for any geodesic $[x,y]$ with $\nhat{x}$ and $\nhat{y}$ 
at least $2\dhat{1,[x,y]} + \Kdisjointa$, and for any point $z$ with
$\nhat{z} \geqslant 2\dhat{1,[x,y]} + \Kdisjointb$, the halfspace
$H(1,z)$ hits at most one of the halfspaces $H(1,x)$ and
$H(1,y)$.
\end{proposition}

We remark that the hypotheses may be restated in terms of the Gromov
product, $(x | y) = \half(\dhat{1,x} + \dhat{1,y} - \dhat{x,y})$,
as in a $\delta$-hyperbolic space the distance from the basepoint $1$
to a geodesic $[x,y]$ is coarsely equivalent to the Gromov product of
$x$ and $y$.

\begin{proof}
We shall choose $\Kdisjointa$ to be $2\Kfellowa + \Kfellowb +
42\delta$, and $\Kdisjointb$ to be $24\delta$. Let $[x,y]$ be a
geodesic from $x$ to $y$, and let $p$ be a closest point on $[x,y]$ to
$1$.  It will be convenient to use the following form of our
assumption that the distance from $1$ to $x$ is larger than twice the
distance from $1$ to $p$.
\begin{equation} \label{equation:assumption} 
\tfrac{1}{2}\dhat{1,x}  \ge \dhat{1,p} + \tfrac{1}{2}\Kdisjointa 
\end{equation}
Let $r$ be the nearest point on $[1,x]$ to $y$, as illustrated below
in Figure \ref{picture9}. We start by showing that $r$ is a bounded
distance from $p$. By Proposition \ref{prop:projection}, the path from
$y$ to $1$ via $p$ is almost a geodesic, so if $t$ is a point on
$[1,p]$ then $\dhat{t,y} \ge \dhat{y,p} + \dhat{p,t} -6\delta$. This
implies that the closest point projection of $y$ to $[x,p] \cup [p,1]$
is distance at most $6\delta$ from $p$. As the path $[1,p] \cup [p,x]$
lies in a $3\delta$-neighbourhood of the geodesic $[1,x]$, as nearest
point projections to close paths are close, Proposition
\ref{prop:close}, this implies that
\begin{equation} \label{equation:rp}
\dhat{r,p} \le 3(3\delta) + 6\delta + 6\delta = 21\delta.
\end{equation}

\begin{figure}[H]
\begin{center}
\epsfig{file=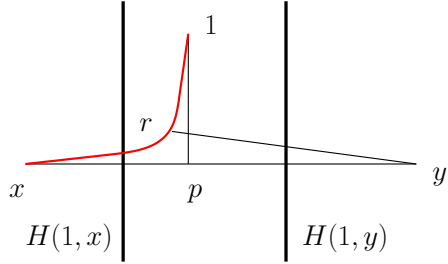, height=100pt}
\end{center}
\caption{Disjoint halfspaces.} \mylabel{picture9}
\end{figure}

In order to use Proposition \ref{prop:fellow} to estimate the size of
the nearest point projection of the halfspace $H(1,x)$ to the geodesic
$[x,y]$ we need to show that that $r$ is closer to $1$ than $x$. As
$r$ lies on a geodesic from $1$ to $x$,
\begin{align*}
\dhat{x,r} & = \dhat{1,x} - \dhat{1,r}.
\intertext{Using our initial assumption, line
  \eqref{equation:assumption}, we may rewrite this as}
\dhat{x,r} & \ge \tfrac{1}{2}\dhat{1,x} + \dhat{1,p} - \dhat{1,r} +
\tfrac{1}{2}\Kdisjointa.
\intertext{Using the bound for $\dhat{r,p}$ from line
  \eqref{equation:rp} and the triangle inequality, we obtain}
\dhat{x,r} & \ge \tfrac{1}{2}\dhat{1,x} + \tfrac{1}{2}\Kdisjointa -
21\delta.
\end{align*}
Therefore we may apply Proposition \ref{prop:fellow}, as we chose
$\Kdisjointa \ge 2\Kfellowa + 42\delta$. This implies that the nearest
point projection of $H(1,x)$ to $[x,y]$ is contained in a
$(\tfrac{1}{2}\dhat{1,x} + \Kfellowb)$-neighbourhood of $x$.

We now show that the nearest point projection of $H(1,x)$ to $[x,y]$
is a definite distance away from $p$. Let $s$ be a point in the
nearest point projection of $H(1,x)$ to $[x,y]$, so we have just shown
that $\dhat{s,x} \le \tfrac{1}{2}\dhat{1,x} + \Kfellowb$. As $s$ lies on the
geodesic $[x,p]$, this implies
\begin{align*}
\dhat{s,p} & \ge \dhat{x,p} - \tfrac{1}{2}\dhat{1,x} - \Kfellowb.
\intertext{Using the triangle inequality applied to going from $1$ to $x$
via $p$, we obtain}
\dhat{s,p} & \ge \dhat{1,x} - \dhat{1,p} - \tfrac{1}{2}\dhat{1,x} - \Kfellowb.
\end{align*}
Our initial assumption \eqref{equation:assumption} now implies that
$\dhat{s,p} \ge \Kdisjointa - \Kfellowb$, which is greater than zero
if $\Kdisjointa > \Kfellowb$. In fact $\dhat{s,p} \ge 14\delta$, as we
chose $\Kdisjointa$ to be $2\Kfellowa + \Kfellowb + 42\delta$.

Our hypotheses are symmetric in $x$ and $y$, so this also implies
identical results for the image of the nearest point projection of
$H(1,y)$ to $[x,y]$. In particular, this implies that the nearest
point projections of $H(1,x)$ and $H(1,y)$ to the geodesic $[x,y]$ are
disjoint.

We now consider a halfspace $H(1,z)$ which intersects both $H(1,x)$
and $H(1,y)$, and show there is an upper bound on $\nhat{z}$.  Assume
that $H(1,z)$ intersects both $H(1,x)$ and $H(1,y)$.  Let $q$ be a
closest point on $[x,y]$ to $z$. Up to relabelling $x$ and $y$, we may
assume that $q \in [p,y]$. This is illustrated in Figure
\ref{picture4} below.

\begin{figure}[H]
\begin{center}
\epsfig{file=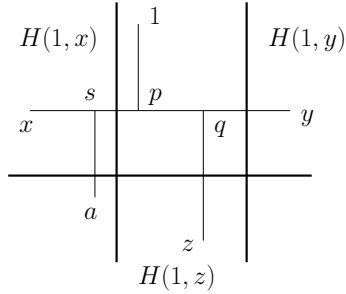, height=110pt}
\end{center}
\caption{Three halfspaces.} \mylabel{picture4}
\end{figure}

Let $a$ be a point lying in the intersection of $H(1,x)$ and $H(1,z)$,
so in particular 
\begin{align*}
\dhat{a,z} & \leqslant \dhat{1,a}. \stepcounter{equation}
\tag{\theequation} \label{equation:az} 
\intertext{Let $s$ be a nearest point on $[x,y]$ to $a$. We showed
  above that $\dhat{s,p} > 14\delta$, and therefore $\dhat{s,q} >
  14\delta$. As $s$ and $q$ are the nearest point projections of $a$
  and $z$ onto the geodesic $[x,y]$, and $\dhat{s,q} > 14\delta$,
  Proposition \ref{prop:double} implies that the path from $a$ to $z$
  via $s$ and $q$ is close to being a geodesic, i.e. $\dhat{a,z} \ge
  \dhat{a,s} + \dhat{s,p} + \dhat{p,q} + \dhat{q,z} - 24\delta$.
  Together with line \eqref{equation:az}, this implies}
\dhat{a,s} + \dhat{s,p} + \dhat{p,q} + \dhat{q,z} - 24\delta &
\leqslant \dhat{1,a}.
\intertext{By the triangle inequality, $\dhat{1,a}
  \le \dhat{1,p} + \dhat{p,s} + \dhat{s,a}$, which gives}
\dhat{a,s} + \dhat{s,p} + \dhat{p,q} + \dhat{q,z} - 24\delta &
\leqslant \dhat{1,p} + \dhat{p,s} + \dhat{s,a}. 
\intertext{Adding $\dhat{1,p} - \dhat{a,s} - \dhat{s,p}$ to both sides, we obtain} 
\dhat{1,p} + \dhat{p,q} + \dhat{q,z} & \leqslant 2\dhat{1,p}
+24\delta. 
\end{align*}
Using the triangle inequality, this shows that $\nhat{z} \leqslant
2\dhat{1,p} +24\delta$. So we may choose $\Kdisjointb$ to be $24\delta$.
\end{proof}

We now show that if $H(1,x)$ is a halfspace far from the origin, then
it is contained in a halfspace $H(1,y)$, with $y$ a bounded distance
closer to the origin, and we find an explicit lower bound on the
distance between $H(1,x)$ and $H(y,1)$. Furthermore, we show that for
any point $a$ in $H(y,1)$, there is a point $b$ such that the halfspace
$H(a,b)$ contains $H(1,x)$, with an explicit lower bound on
$\dhat{a,b}$. Similarly, we show that for any point $b$ in $H(1,x)$,
there is a point $a$ such that $H(y,1) \subset H(b,a)$, again with an
explicit lower bound on $\dhat{a,b}$.

\begin{proposition} \mylabel{prop:nested}
There are constants $\Knest$ and $\Knesta$, which only depend on
$\delta$, such that for any positive number $A$, and any point $x$
with $\nhat{x} \ge \Knest + 4A$, there is a point $y$ on $[1,x]$ with
$\nhat{y} = \nhat{x} - 2\Knesta - 2A$, such that $H(1,x) \subset
H(1,y)$. Furthermore, for any point $a$ in $H(y,1)$ there is a point
$b$ in $H(1,x)$, with $\dhat{a,b} \ge 2A$, and $H(1,x) \subset
H(a,b)$. Also, for any point $b$ in $H(1,x)$, there is a point $a$ in
$H(y,1)$ with $\dhat{a,b} \ge 2A$, and $H(y,1) \subset H(b,a)$.
\end{proposition}

\begin{proof}
We shall choose $\Knesta$ to be $98\delta + 2\Kstability$ and $\Knest$
to be $2\Knesta + 6\delta$.
We first show that $H(1,x) \subset H(1,y)$. Let $p$ be a point in the
nearest point projection of $H(y,1)$ to $[1,x]$, and let $q$ be a
point in the nearest point projection of $H(1,x)$ to $[1,x]$. Applying
Proposition \ref{prop:half} to the halfspace $H(1,y)$ gives 
\begin{align}
\dhat{1,p} & \le \half \nhat{y} + 3\delta. \notag
\intertext{Using our assumption that $\nhat{y} = \nhat{x} - 2\Knesta -
  2A$ implies}
\dhat{1,p} & \le \half \nhat{x} + 3\delta - A - \Knesta. \label{equation:nested1}
\end{align}
Now applying Proposition \ref{prop:half} to the halfspace $H(1,x)$, we
obtain
\begin{equation}
\dhat{q,1} \ge \half \nhat{x} - 3\delta. \label{equation:nested2}
\end{equation}
Comparing \eqref{equation:nested1} and \eqref{equation:nested2} shows
that if $\Knesta \ge 13\delta$, then $\dhat{p,q} \ge 7\delta$, and so
$H(1,x)$ and $H(y,1)$ are disjoint, which implies $H(1,x) \subset
H(1,y)$, as required.

We now prove the second statement in Proposition \ref{prop:nested}.

\begin{claim}
For any point $a$ in $H(y,1)$ there is a point
$b$ in $H(1,x)$, with $\dhat{a,b} \ge 2A$, and $H(1,x) \subset
H(a,b)$.
\end{claim}

\begin{proof}
Given a point $a$ in $H(y,1)$ we will choose a point $b$ in the
geodesic $[1,x]$, sufficiently far from $a$, and then show that
$H(1,x) \subset H(a,b)$ by showing that the nearest point projections
of $H(1,x)$ and $H(b,a)$ to $[1,x]$ are sufficiently far apart.

Let $a$ be a point in the halfspace $H(y,1)$, let $p$ be a nearest
point projection of $a$ to $[1,x]$ and let $b$ be a point on the
geodesic $[1,x]$ distance $\nhat{p} + \Knesta$ from $x$. This is
illustrated below in Figure \ref{picture20}.

\begin{figure}[H]
\begin{center}
\epsfig{file=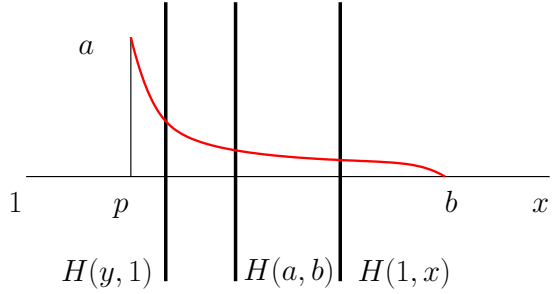, height=110pt}
\end{center}
\caption{The halfspace $H(1,x)$ is contained in the halfspace
  $H(a,b)$.} \mylabel{picture20}
\end{figure}

We first find a lower bound for the distance between $a$ and $b$.
Using Proposition \ref{prop:projection}, which says that the path
$[a,p] \cup [p,b]$ is contained in a $3\delta$-neighbourhood of
$[a,b]$, and the triangle inequality, we obtain
\begin{equation} \label{equation:ab}
\dhat{a,b} \ge \dhat{p,b} - 6\delta. 
\end{equation}
Let $z$ be the midpoint of $[1,x]$. Then $\dhat{p,z} \ge \Knesta + A -
3\delta$, by Proposition \ref{prop:half}, and our assumption on the
distance between the halfspaces $H(1,x)$ and $H(1,y)$. Our choice of
$b$ then implies $\dhat{z,b} \ge A - 3\delta$. Therefore $\dhat{p,b}
\ge \Knesta + 2A - 12\delta$, and this is at least $2A$, as required, as
we chose $\Knesta \ge 12\delta$.

We now show that the halfspace $H(1,x)$ is contained in the halfspace
$H(a,b)$. As nearest point projection is coarsely well-defined,
Proposition \ref{prop:npp}, it suffices to show that the nearest
point projections of $H(1,x)$ and $H(b,a)$ to $[1,x]$ are distance at
least $7\delta$ apart. By Proposition \ref{prop:half}, if $r$ is a
point in the nearest point projection of $H(1,x)$ to $[1,x]$ then
\begin{equation} \label{equation:1x estimate}
\dhat{r,x} \le \half \dhat{1,x} + 3\delta. 
\end{equation}

Similarly, the nearest point projection of $H(b,a)$ to $[a,b]$ is
distance at least $\half \dhat{a,b} - 3\delta$ from $b$. Let $s$ be a
closest point on $[a,b]$ to $p$. By Proposition \ref{prop:stability},
if $r$ is a point in the nearest point projection of $H(b,a)$ to the
subgeodesic $[s,b]$ of $[a,b]$, then
\begin{align*}
\dhat{r,b} & \ge \half \dhat{a,b} - 3\delta - \Kstability.
\intertext{By Proposition \ref{prop:projection}, $\dhat{s,p} \le
  3\delta$, and the path $[s,p] \cup [p,b]$ is contained in a
  $3\delta$-neighbourhood of $[s,b]$. Therefore, as nearest point
  projections to close paths are close, Proposition \ref{prop:close},
  this implies that if $r$ is a nearest point in $[s,p] \cup [p,b]$ to
  $H(b,a)$ then}
\dhat{r,b} & \ge \half \dhat{a,b} - 18\delta - \Kstability.
\intertext{Nearest point projection onto the path $[s,p] \cup [p,b]$
  is the same as nearest point projection onto the path $[p,s] \cup
  [s,p] \cup [p,b]$, and this latter path is contained in a
  $3\delta$-neighbourhood of $[p,b]$, so again applying Proposition
  \ref{prop:close}, if $r$ lies in the nearest point projection of
  $H(b,a)$ to $[p,b]$ then}
\dhat{r,b} & \ge \half \dhat{a,b} - 33\delta - \Kstability. 
\intertext{Using \eqref{equation:ab}, and the fact that $[p,b]$ is a geodesic
subsegment of $[1,x]$, we obtain}
\dhat{r,b} & \ge \half \dhat{1,x} - \half \dhat{1,p} - \half
\dhat{b,x} - 36\delta - \Kstability.
\intertext{The points $r$ and $b$ lie on the geodesic $[1,x]$, and
  $\dhat{b,x} = \dhat{1,p} + \Knesta$ from $x$, so this implies}
\dhat{r,x} & \ge \half \dhat{1,x} + \half \Knesta - 36\delta -
\Kstability. \stepcounter{equation} \tag{\theequation}
\label{equation:ba estimate}
\end{align*}
Comparing \eqref{equation:1x estimate} with \eqref{equation:ba
  estimate} shows that the distance between the projections of the two
halfspaces to $[1,x]$ is at least $7\delta$, as $\Knesta \ge 92\delta
+ 2 \Kstability$. Therefore the nearest point projections of $H(1,x)$
and $H(b,a)$ are sufficiently far apart, and so $H(1,x) \subset
H(a,b)$, as required.
\end{proof}

We now prove the final statement, by an analogous argument to the one
above, though unfortunately not exactly the same, as the picture is
not completely symmetric.  

\begin{claim}
For any point $b$ in $H(1,x)$, there is a point $a$ in
$H(y,1)$ with $\dhat{a,b} \ge 2A$, and $H(y,1) \subset H(b,a)$.
\end{claim}

\begin{proof}
Let $b$ be a point in the halfspace $H(1,x)$, and let $p$ be a nearest
point projection of $b$ to $[1,x]$. Let $a$ be a point on the
geodesic $[1,x]$ distance $\dhat{p,x} - 2A - \Knesta$ from $1$, and
this is greater than zero, as $\dhat{1,x} \ge 4A + 2\Knesta +
6\delta$. This is illustrated below in Figure \ref{picture21}.

\begin{figure}[H]
\begin{center}
\epsfig{file=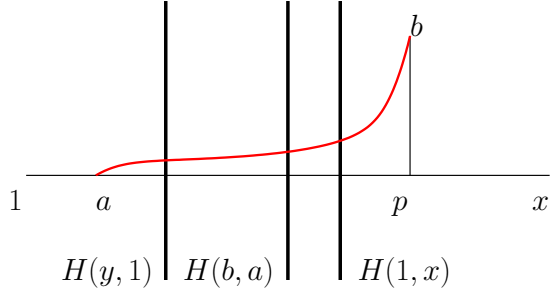, height=110pt}
\end{center}
\caption{The halfspace $H(y,1)$ is contained in the halfspace
  $H(b,a)$.} \mylabel{picture21}
\end{figure}

We start by showing that $\dhat{a,b} \ge 2A$. Using Proposition
\ref{prop:projection}, which says that the path $[a,p] \cup [p,b]$ is
contained in a $3\delta$-neighbourhood of $[a,b]$, and the triangle
inequality, we obtain
\begin{equation} \label{equation:ab2}
\dhat{a,b} \ge \dhat{a,p} - 6\delta.
\end{equation}
Proposition \ref{prop:half}, $\dhat{p,x} \le \half \dhat{1,x} +
3\delta$, and our choice of $a$, implies that $\dhat{1,a} \le \half
\dhat{1,x} + 3\delta - 2A - \Knesta$. Therefore $\dhat{a,p} \ge 2A +
\Knesta - 6\delta$, and so $\dhat{a,b} \ge 2A$ as $\Knesta \ge
12\delta$.

We wish to show that $H(y,1) \subset H(b,a)$, and as before it
suffices to show that the nearest point projections of the halfspaces
$H(y,1)$ and $H(a,b)$ to the geodesic $[1,x]$ are distance at least
$7\delta$ apart, as nearest point projection is coarsely well defined,
Proposition \ref{prop:npp}. Let $r$ be a point in the nearest point
projection of $H(y,1)$ to $[1,x]$ which is furthest from $1$. By
Proposition \ref{prop:half}, $\dhat{1,r} \le \half \dhat{1,y} +
3\delta$. Also, as $\dhat{H(y,1),H(1,x)} \ge A + \Knesta$, again using
Proposition \ref{prop:half}, this implies $\dhat{x,y} \ge 2A +
2\Knesta + 6\delta$.  Therefore
\begin{equation}
\dhat{1,r} \le \half \dhat{1,x} - A - \Knesta + 6\delta.
\label{equation:1y estimate}
\end{equation}
Now applying Proposition \ref{prop:half} to the halfspace $H(a,b)$,
implies that if $r$ is a point in the nearest point projection of
$H(a,b)$ to $[a,b]$, then
\begin{align*}
\dhat{a,r} & \ge \half \dhat{a,b} - 3 \delta.
\intertext{Let $s$ be a closest point on $[a,b]$ to $p$. As $[a,s]$ is
  a subgeodesic of $[a,b]$, we may use Proposition
  \ref{prop:stability}, which implies that if $r$ is a point in the
  nearest point projection of $H(a,b)$ to $[a,s]$, then}
\dhat{a,r} & \ge \half \dhat{a,b} - 3\delta - \Kstability.
\intertext{By Proposition \ref{prop:projection}, $\dhat{s,p} \le
  3\delta$, and the path $[a,p] \cup [p,s]$ is contained in a
  $3\delta$-neighbourhood of $[a,s]$. Therefore, as nearest point
  projections to close paths are close, Proposition \ref{prop:close},
  this implies that if $r$ is a nearest point in $[a,p] \cup [p,s]$ to
  $H(a,b)$ then }
\dhat{a,r} & \ge \half \dhat{a,b} - 18\delta - \Kstability.
\intertext{Nearest point projection onto the path $[a,p] \cup [p,s]$ is the same
as nearest point projection onto the path $[a,p] \cup [p,s] \cup
[s,p]$, and this latter path is contained in a $3\delta$-neighbourhood
of $[a,p]$, so again applying Proposition \ref{prop:close}, if $r$
lies in the nearest point projection of $H(a,b)$ to $[a,p]$ then }
\dhat{a,r} & \ge \half \dhat{a,b} - 33\delta - \Kstability.
\intertext{
Using \eqref{equation:ab2}, and the fact that $[a,p]$ is a geodesic
subsegment of $[1,x]$, we obtain
}
\dhat{1,r} & \ge \half \dhat{1,p} + \half \dhat{1,a} - 36\delta - \Kstability.
\intertext{As we chose $a$ such that $\dhat{1,a} = \dhat{p,x} - 2A -
  \Knesta$, this shows}
\dhat{1,r} & \ge \half \dhat{1,x} -A - \half \Knesta - 36\delta -
\Kstability. \stepcounter{equation} \tag{\theequation}
\label{equation:ab estimate}
\end{align*} 
Comparing \eqref{equation:ab estimate} with \eqref{equation:1y
  estimate} shows that the distance between the projections of the two
halfspaces to $[1,x]$ is at least $7\delta$, as $\Knesta \ge 98\delta +
2 \Kstability$. Therefore the nearest point projections of $H(y,1)$
and $H(a,b)$ are sufficiently far apart, and so $H(y,1) \subset
H(b,a)$, as required.
\end{proof}

This completes the proof of Proposition \ref{prop:nested}.
\end{proof}

Finally, we prove a result about a slightly more general definition of
halfspaces. Let $H(1,x;C) = \{ y \mid \dhat{y,x} \le \dhat{1,x} + C
\}$. We allow $C$ to be negative, and if $C$ is zero, this recovers
the standard definition of a halfspace, and we shall continue to write
$H(1,x)$ to mean $H(1,x;0)$. We now show that a coarse halfspace is
contained in a halfspace.

\begin{proposition} \mylabel{prop:coarse}
The halfspace $H(1,x;C)$ is contained in the halfspace $H(1,y)$,
for $\nhat{y} = \nhat{x} - C - 9\delta$, if $0 \le C \le \nhat{x} - 9\delta$.  
\end{proposition}

\begin{proof}
Let $z \in H(1,x;C)$, and let $p$ be a nearest point to $z$ in
$[1,x]$, so $\dhat{z,x} \le \dhat{z,1} + C$. By the triangle
inequality, $\dhat{1,z} \le \dhat{1,p} + \dhat{p,z}$, therefore
\begin{align*}
\dhat{z,x} & \le \dhat{1,p} + \dhat{p,z} + C.
\intertext{The path from $z$ to $x$ via $p$ is almost a geodesic,
  Proposition \ref{prop:projection}, so $\dhat{z,x} \ge \dhat{z,p} +
  \dhat{p,x} - 6\delta$. This implies}
\dhat{p,x} & \le \dhat{1,p} + C + 6\delta.
\intertext{As $p$ lies on the geodesic $[1,x]$, we may rewrite this as}
\dhat{p,x} & \le \half \dhat{1,x} + \half C + 3\delta.
\end{align*}
By Proposition \ref{prop:half}, this implies that $z$ is contained in
the halfspace $H(1,y)$, for $\nhat{y} = \nhat{x} - C - 9\delta$, as
required.
\end{proof}

\section{Non-elementary semi-groups} \label{section:semigroups}

We do not assume that our random walk is symmetric, so it will be
convenient to know that if the group generated by the support of $\mu$
is non-elementary, then the semi-group generated by the support of
$\mu$ contains a pair of independent \pA elements. This follows from
well-known results on the structure of subgroups of the mapping class
group, which we now briefly review. We will use the definitions and
results of Ivanov \cite{ivanov}, although the results we obtain could
also be deduced from work of McCarthy \cite{mccarthy} and Birman,
Lubotzky and McCarthy \cite{blm}.

We say an element $h$ of the mapping class group is \emph{pure}, if
there is a disjoint collection of simple closed curves $\s(h)$ which
are fixed individually by $h$, such that each complementary component
of $\s(h)$ is also preserved, and furthermore $h$ acts on each
complementary component as either a pseudo-Anosov element or the
identity.  If the collection of simple closed curves $\s(h)$ has the
property that no simple closed curve with non-zero intersection number
with $\s(h)$ is fixed by $h$, then $\s(h)$ is called a \emph{canonical
  reduction system} for $h$.  If $h$ is not pure, then we define the
canonical reduction set $\s(h)$ to be the canonical reduction set of
some pure power of $h$. Note that $\s(fgf^{-1}) = f\s(g)$. The
canonical reduction system of a periodic or \pA element of the mapping
class group is empty. Ivanov shows that given reducible elements of
the mapping class group $f$ and $g$, there is a product of $f$ and $g$
whose canonical reduction set is the intersection of the canonical
reduction sets of $f$ and $g$.

\begin{lemma} \cite{ivanov}*{Lemma 5.1}
\label{lemma:reduction}
Let $f$ and $g$ be reducible elements of the mapping class group. Then
there are positive numbers $k$ and $l$, such that the canonical
reduction set of $f^k g^l$ is the intersection of the canonical
reduction sets for $f$ and $g$.
\end{lemma}

A \pA element $f$ acts on $\PMF$ with simple dynamics, there are a
pair of fixed points, which we shall refer to as the stable fixed
point $\l^+_f$ and the unstable fixed point $\l^-_f$. The stable fixed
point is attracting, i.e. every point in $\PMF \setminus \l^-_f$
converges to $\l^+_f$ under iteration by $f$, and the unstable fixed
point is repelling, i.e. every point in $\PMF \setminus \l^+_f$
converges to $\l^-_f$ under iteration by $f^{-1}$.  Ivanov
\cite{ivanov} shows that a pair of \pA elements either have the same
pair of fixed points, or they have disjoint fixed points.

\begin{lemma}\cite{ivanov}*{Lemma 5.11}
\label{lemma:endpoints}
If $f$ and $g$ are \pA elements, then
either $\fix{f} = \fix{g}$, or $\fix{f} \cap \fix{g}$ is empty.
\end{lemma}

We will also make use of the following observation from Ivanov
\cite{ivanov}.

\begin{lemma} \cite{ivanov}*{Chapter 11, exercise 3(a)} 
\label{lemma:ivanov}
Let $f$ be a \pA element, with fixed points $\l^\pm_f$. Let $g$ be any
element such that $g(\l^+_f) \not = \l^-_f$. Then $f^ng$ is \pA for all
sufficiently large $n$.
\end{lemma}

We now show that if a non-elementary subgroup $M$ is generated by a
semi-group $M^+$, then we can find a pair of independent \pA elements
in $M^+$.

\begin{lemma}
Let $M^+$ be a semi-group in the mapping class group which generates a
non-elementary group $M$. Then $M^+$ contains a pair of independent
\pA elements.
\end{lemma}

\begin{proof}
The group $M$ contains two independent \pA elements.  We now show that
we can find a pair of independent \pA elements which in fact lie in
$M^+$.

Suppose that $M^+$ does not contain a pair of independent \pA
elements, but does contain at least one \pA element.  By Lemma
\ref{lemma:endpoints}, if $M^+$ contains \pA elements, they must all
have common endpoints. Let $f$ be a \pA element of $M^+$, with fixed
points $\l^+_f$ and $\l^-_f$. We will now show that all periodic and
reducible elements of $M^+$ preserve the fixed point of $f$.

Suppose $g$ is a periodic element of $M^+$. If $g$ does not preserve
the fixed points of $f$, then $gfg^{-1}$ is a \pA element with
distinct fixed points, and furthermore lies in $M^+$, as $g^{-1}$ may
be written as a positive power of $g$. So all periodic elements of
$M^+$ preserve the fixed points of $f$.

Suppose $g$ is a reducible element of $M^+$. At least one of $g$ or
$g^2$ has the property that the image of $\l^+_f$ is not $\l^-_f$, so
possibly after replacing $g$ by its square, we may assume that $g$ has
this property.  Then by Lemma \ref{lemma:ivanov}, the element $f^ng$
is \pA for all sufficiently large $n$. In particular $gf^ngg^{-1} =
gf^n$ is \pA, and lies in $M^+$. As we have assumed that $M^+$
contains no independent \pA elements, $gf^n$ must have the same fixed
points as $f$, but this implies that $g$ preserves the fixed points of
$f$.

We have shown that if $M^+$ contains \pA elements, then they must all
have common fixed points, and all other elements of $M^+$ preserve the
set of fixed points. But then the inverses of all elements of $M^+$
preserve the set of fixed points, so in in fact the entire group $M$
preserves the pair of fixed points, which contradicts the fact that
$M$ is non-elementary.  So we may assume that the semi-group $M^+$
contains no \pA elements.

If there is a disjoint collection of essential simple closed curves in
the surface $\S$ fixed by all elements of $M^+$, then this disjoint
collection of simple closed curves is also fixed by all inverses of
elements in $M^+$, and hence by the whole group $M$. As $M$ is
non-elementary it fixes no disjoint collection of simple closed curves
in $\S$, so this implies that the semi-group $M^+$ also fixes no
disjoint collection of simple closed curves in the surface $\S$.
However, it could \emph{a priori} be the case that all reducible
elements of $M^+$ have common simple closed curves in their canonical
reduction sets, but these are not preserved by elliptic elements of
$M^+$.

Let $\rho$ be the intersection of the canonical reduction sets of all
reducible elements of $M^+$, i.e. $\rho = \bigcap \{ \s(g) \mid g \in
M^+, g \text{ reducible} \}$. Lemma \ref{lemma:reduction} shows that
given reducible elements of the mapping class group $f$ and $g$, there
is an element of the semi-group generated by $f$ and $g$ whose
canonical reduction set is the intersection of the canonical reduction
sets for $f$ and $g$. This implies that there is a reducible element
$r \in M^+$ with $\s(r) = \rho$.

Suppose $g \in M^+$ does not preserve $\rho$. Let $k$ be the smallest
positive integer such that $g^k$ is pure. If $g$ is periodic, then $k$
is the order of $g$, and $\s(g)$ is empty. If $g$ is reducible, then
$\rho \subset \s(g)$. By Lemma \ref{lemma:ivanov}, for all
sufficiently large positive integers $l$, the element $r^l g^k$ is \pA
on each component of $\S \setminus \rho$ on which either $r$ or $g^k$
is \pA. In particular, $\s(r^l g^k) = \rho$, and $r^l g^k$ is not
periodic. As we have assumed $M^+$ contains no \pA elements, $r^l g^k$
must in fact be reducible.

Consider conjugating $r^l g^k$ by $g$, i.e. $gf^lg^k g^{-1} =
gf^lg^{k-1}$. This gives a reducible element of $M^+$ with canonical
reduction set $\s(gf^lg^{k-1}) = g \s(f^l g^k) = g(\rho) \not = \rho$.
So we can construct an element of $M^+$ with canonical reduction set
$\rho \cap g(\rho)$, which is strictly smaller than $\rho$, which
contradicts the fact that $\rho$ is the intersection of all of the
canonical reduction sets of all of the reducible elements of $M^+$. So
$\rho$ is preserved by all elements of $M^+$, and hence by the inverse
of all elements of $M^+$, and therefore by the entire group $M$.  But
then the group $M$ is reducible, not non-elementary, a contradiction.
\end{proof}

\section{Exponential decay} \label{section:decay}

In this section we show that the harmonic measure of halfspaces
$H(1,x)$ decay exponentially with their distance from the origin. As
the number of disjoint halfspaces $H(1,x)$ grows exponentially in
$\nhat{x}$, this is clearly true on average, but in this section we
show that the measure of all halfspaces decays as $L^{\nhat{x}}$, for
some constant $L < 1$ which does not depend on the choice of
halfspaces. We then show that there is a similar estimate for the
convolutions measures $\mun{n}$, i.e.  $\mun{n}(H(1,x)) \le
QL^{\nhat{x}}$, for some constant $Q$, and where $L < 1$ is the same
constant as for $\nu$.

We will need to use the fact that the group generated by the support
of the random walk is non-elementary, and it will be convenient to use
the fact that non-elementary subgroups have free subgroups that act on
the complex of curves in a similar manner to Schottky groups acting on
hyperbolic space.

\begin{definition}
We say that a pair of elements $a,b \in G$ are a \emph{Schottky pair}
if there are disjoint halfspaces $A^+, A^-,B^+$ and $B^-$ in the
relative space $\Grel$, such that $a(\Grel \setminus A^-) \subset A^+,
a^{-1}(\Grel \setminus A^+) \subset A^-, b(\Grel \setminus B^-)
\subset B^+$ and $b^{-1}(\Grel \setminus B^+) \subset B^-$.  We will
refer to a choice of halfspaces $A^\pm, B^\pm$, with the properties
above, as Schottky halfspaces for the Schottky pair $a,b$.
\end{definition}

\begin{figure}[H]
\begin{center}
\epsfig{file=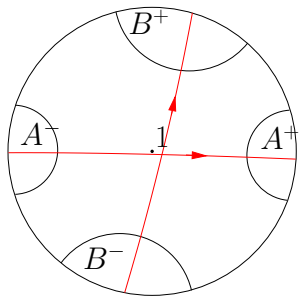, height=110pt}
\end{center}
\caption{Halfspaces for a Schottky pair.} \mylabel{picture7}
\end{figure}

This is illustrated in in Figure \ref{picture7} above. The red lines
correspond to axes for $a$ and $b$.  The subgroup generated by a
Schottky pair is a free group with two generators, in which all
non-identity elements are \pA. A non-elementary subgroup of the
mapping class group contains a pair of \pA elements with distinct
endpoints in $\PMF$, and these correspond to distinct points in the
Gromov boundary of the complex of curves by work of Klarreich
\cite{klarreich}. Therefore there are integers $p$ and $q$ such that
$a^p$ and $b^q$ are a Schottky pair. Furthermore, we may replace $p$
and $q$ with any larger integers.

We first show that the harmonic measure of a Schottky halfspace is
non-zero, for any Schottky pair contained in the support of $\mu$.

\begin{proposition} \mylabel{prop:non-zero measure}
Let $A^+$ be a Schottky halfspace, whose corresponding pseudo-Anosov
element lies in the subgroup generated by the support of $\mu$. Then
$\nu(\overline{A^+}) > 0$.
\end{proposition}

\begin{proof}
Suppose $\nu(\overline{A^+}) = 0$. The measure $\nu$ is
$\mun{n}$-invariant, i.e. $\nu(X) = \sum \mun{n}(g) \nu(g^{-1}X)$, so
if $\mun{n}(g) \not = 0$, then $\nu(g^{-1}\overline{A^+}) = 0$. As $a$
lies in the support of $\mu$, there is an $n$ such that $\mun{n}(a) >
0$. This implies that $\nu(a^{-k}\overline{A^+}) = 0$ for all $k$.
However, the complement of the union of the $a^{-k}A^+$ is the
unstable fixed point of $a$, which has measure zero as $\nu$ is
non-atomic. This means that the total measure $\nu(\PMF)$ is zero, a
contradiction.
\end{proof}

We now show that the harmonic measure of $H(1,x)$ is bounded away
from $1$, at least for $\nhat{x}$ sufficiently large.

\begin{proposition} \mylabel{prop:bound}
Let $\mu$ be a probability distribution on $G$ whose support generates
a non-elementary subgroup, and let $\nu$ be the corresponding harmonic
measure. Then there are constants $\Kbound$ and $\e > 0$ such that
$\nu(\overline{H(1,x)}) \leqslant 1 - \e$, for all $x$ with $\nhat{x}
\geqslant \Kbound$. The constant $\Kbound$ only depends on the support
of $\mu$. The constant $\e$ depends on $\mu$.
\end{proposition}

\begin{proof}
Let $M$ be the subgroup generated by the support of $\mu$. As this
subgroup is non-elementary, it contains a Schottky pair, $\{ a,b\}$.
As $a$ and $b$ are pseudo-Anosov, the sequences $a^k$ and $b^k$ are
quasi-geodesic, with distinct endpoints, so possibly after replacing
$a$ and $b$ by positive powers $a^r$ and $b^s$, we may assume that
$\{H(1,a^{\pm 1}), H(1,b^{\pm 1}) \}$ are a collection of Schottky
halfspaces for $\{a,b\}$, and $\min \{ \nhat{a}, \nhat{b }\} >
2\dhat{1,[a,b]} + \Kdisjointa$.  By Proposition \ref{prop:disjoint} there is a
$\Kbound = 2\dhat{1,[a,b]} + \Kdisjointb$ such that any halfspace $H(1,x)$
with $\nhat{x} \geqslant \Kbound$ intersects at most one of the Schottky
halfspaces. In particular this means that the half space $H(1,x)$ is
disjoint from at least three of the Schottky half spaces, so we may
choose $\e$ to be the minimum of the sum of any three of the harmonic
measures of the four Schottky halfspaces.
\end{proof}

We now show that the harmonic measure of halfspaces $H(1,x)$ decays
exponentially in $\nhat{x}$ as $x$ varies, at least for $x$ with
$\nhat{x}$ sufficiently large.

\begin{lemma} \mylabel{lemma:exponential}
Let $\mu$ be a probability distribution on $G$ whose support is
relatively bounded and which generates a
non-elementary subgroup, and let $\nu$ be the corresponding harmonic
measure. Then there are constants $\Kexp$ and $L<1$, such that if
$\nhat{x} \geqslant \Kexp$ then $\nu(\overline{H(1,x)}) \leqslant L^{\nhat{x}}$. Both
$\Kexp$ and $L$ depend on $\mu$.
\end{lemma}

\begin{proof}
As the subgroup generated by the support of $\mu$ is non-elementary,
the subgroup contains a Schottky pair, so by Proposition
\ref{prop:bound}, there are constants $\Kbound$ and $\e$ such that 
$\nu(\overline{H(1,x)}) \leqslant 1-\e$, for $\nhat{x} \ge \Kbound$. 

Before considering more general random walks, we explain the argument
in the case of the nearest neighbour random walk on the Cayley graph.
Consider a pair of nested halfspaces $H(1,x_2)$ and $H(1,x_2)$, where
$\nhat{x_2} \ge 2\Kbound + \Knest$, and $x_1 \in [1,x_2]$ with
$\nhat{x_1} = \nhat{x_2} - \Kbound - 2\Knesta$. Given a halfspace
$H(1,x_1)$, we will refer to the set of points which are equal
distance from both $1$ and $x_1$ as the \emph{equidistant set} for
$H(1,x_1)$, which we shall denote by $E(1,x_1)$. We remark that
\emph{a priori} it may be the case that equidistant set $E(1,x_1) =
H(1,x_1) \cap H(x_1,1)$ has measure strictly larger than zero.  Any
sample path that converges into the limit set of $H(1,x_2)$, must pass
through the equidistant set of $H(1,x_1)$. So we can use the formula
for conditional probability, conditioning on the location at which the
sample path first hits the equidistant set of $H(1,x_1)$.  By
Proposition \ref{prop:nested}, for any point $a \in E(1,x_1)$, the
halfspace $H(1,x_2)$ is contained in a halfspace $H(a,b)$, with
$\dhat{a,b} \geqslant \Kbound$, so the probability that you converge
into the limit set of $H(1,x_2)$ from any point of $E(1,x_1)$ is at
most $1 - \e$, so, in particular, the probability that you converge
into the limit set of $H(1,x_2)$ is at most $1 - \e$.  We may apply
this argument to any nested sequence of halfspaces, each distance at
least $2\Kbound + \Knest$ apart, so the probability that you converge
into the limit set of a halfspace $H(1,x)$ with $\nhat{x} \ge
(2\Kbound + \Knest)n$ is most $(1-\e)^{(n-1)}$.

For more general probability distributions, we need to replace the
condition of hitting the equidistant set of a halfspace, with the
condition of hitting some suitably large neighbourhood of the
equidistant set.

Let $H_i$ be a sequence of halfspaces $H(1,x_i)$, where the $x_i$ lie
on a common geodesic starting at $1$, with $\nhat{x_{i+1}} \ge
\nhat{x_i} + 2\Kbound + \Knest$ for each $i$.  Consider a sample path
which converges to a boundary point $\l$ in $\overline H_k$. Such a
sample path must have infinitely many elements in $H_1$.  If such a
sample path does not hit $H_1 \setminus H_k$ then it must first hit
$H_1$ inside $H_k$, after a jump of distance at least $k(2\Kbound +
\Knest)$. Pick $k$ such that $k(2\Kbound + \Knest)$ is larger than the
relative diameter of the support of $\mu$.

For a sample path $w$, let $F(w)$ be the group element corresponding
to location at which $\w$ first hits $H_1$. This is well defined for
sample paths which converge to $\l \in \overline H_2$.  All sample
paths which converge to $\l \in \overline H_k$ hit $H_1 \setminus
H_k$, so conditioning on $F(\w)$ gives
\begin{align*}
  \P(\l(\w) \in \overline H_{k+1} \mid \l(\w) \in \overline H_{k} ) &
  \leqslant \sum_{g \in H_1 \setminus H_k} \P(F(\w) =
  g) \
  \nu(g^{-1}\overline{H_{k+1}})
\intertext{By Proposition \ref{prop:nested}, for any point $g$ in
  $H_1 \setminus H_k$, there is a point $b$ in $H_{k+1}$ such that
  $H_{k+1} \subset H(g,b)$, and $\dhat{g,b} \ge \Kbound$. Therefore,
  by Proposition \ref{prop:bound}, $\nu(g^{-1} \overline{H_{k+1}}) \le 1-\e$, so}
\P(\l(\w) \in \overline H_{k+1} \mid \l(\w) \in \overline H_{k} ) & \leqslant 1 - \e
\end{align*}
This implies that $\nu(\overline{H_{k+1}}) \leqslant
\nu(\overline{H_k})(1-\e)$, so by induction $\nu(\overline{H_{k+l}})
\leqslant (1-\e)^{l+1}$. Therefore $\nu(\overline{H_n}) \leqslant
L^n$, for $\nhat{x} \ge \Kexp$, where we may choose $L =
(1-\e)^{1/(k+1)}$ and $\Kexp = k(2\Kbound + \Knest)$.
\end{proof}

The measure $\nu$ is the weak limit of the measures $\mun{n}$, so one
may hope there is exponential decay for halfspaces for the
$\mun{n}$-measures, at least for large $n$. In fact, we now show that
there is an upper bound for $\mun{n}(H(1,x))$ which decays
exponentially in $\nhat{x}$, for all $n$. 

\begin{lemma} \mylabel{lemma:mu-n-decay}
There are constants $\Kmun, L<1$ and $Q$, which only depend on
$\delta$ and $\mu$, such that if $H(1,x)$ is a halfspace with
$\nhat{x} \geqslant \Kmun$, then $ \mun{n}(H(1,x)) \le Q
L^{\nhat{x}}$.
\end{lemma}

\begin{proof}
We shall choose $\Kmun = \Knest + 2\Kbound + \Kexp$.  Let $H(1,x)$ be
a halfspace with $\nhat{x} \geqslant 2\Kbound + \Knest$.  Then, by
Proposition \ref{prop:nested}, we may choose $y \in [1,x]$ with
$\nhat{y} = \nhat{x} - \Kbound - 2\Knesta$, such that $H(1,x) \subset
H(1,y)$, and, furthermore, for any $b \in H(1,x)$, there is a point
$a$ in $H(y,1)$ such that the halfspace $H(y,1)$ is contained in the
halfspace $H(b,a)$, with $\dhat{a,b} \geqslant \Kbound$.  By
Proposition \ref{prop:bound}, the harmonic measure of
$\overline{H(y,1)}$ viewed from any $b \in H(1,x)$ is at most $1-\e$,
i.e. if $\nu_b$ is the harmonic measure induced by a random walk
starting at $b$ instead of $1$, then $\nu_b(\overline{H(y,1)})
\leqslant 1-\e$.  At time $n$, the halfspace $H(1,x)$ has measure
$\mun{n}(H(1,x))$, so the proportion of these sample paths which
converge into the limit set of $H(y,1)$ is at most $1-\e$. Therefore
at least $\e$ of these sample paths converge into the complement of
the limit set of $H(y,1)$, which is contained in the limit set of
$H(1,y)$, i.e.  $\nu(\overline{H(1,y)}) \geqslant \e \mun{n}(H(1,x))$.
This implies
\begin{align*}
\mun{n}(H(1,x)) & \leqslant \tfrac{1}{\e}\nu(\overline{H(1,y)}). \\
\intertext{We chose $\nhat{x} \ge \Knest + 2\Kbound + \Kexp$, and
  $\nhat{y} = \nhat{x} - \Kbound - 2\Knesta$, so we may apply the
  exponential decay bounds from Lemma \ref{lemma:exponential} to obtain}
\mun{n}(H(1,x)) & \le \tfrac{1}{\e}L^{\nhat{x} - \Kbound - 2\Knesta}.
\end{align*}
Therefore we may choose $Q$ to be
$\tfrac{1}{\e}L^{-\Kbound-2\Knesta}$, and this depends on $\delta$ and
$\mu$, but not on $x$ or $n$.  Therefore we have shown that there is a
constant $Q$ such that $\mun{n}(H(1,x)) \leqslant Q L^{\nhat{x}}$, for
all $n$, as long as $\nhat{x} \geqslant \Kmun = \Knest + 2\Kbound +
\Kexp$. The constant $L$ may be chosen to have the same value as the
constant $L$ from Lemma \ref{lemma:exponential}.
\end{proof}

We remark that for the nearest neighbour random walk, for small $n$
and $\nhat{x}$ large, $\mun{n}(H(1,x))$ will be zero until $n$ is at
least $\nhat{x}$, so $\mun{n}(H(1,x))$ need not be monotonically decreasing
in $n$.

\section{Linear progress} \label{section:linear progress}

We may now complete the proof of the main theorem.

\begin{theorem:main}
Let $G$ be the mapping class group of an orientable surface of finite
type, which is not a sphere with three or fewer punctures, and
consider the random walk generated by a probability
distribution $\mu$, whose support is bounded in the relative metric
and which generates a non-elementary subgroup
of the mapping class group, and which has finite first
moment. Then there is a constant $\ell > 0$ such that $\lim_{n \to \infty}
\tfrac{1}{n}\nhat{w_n} = \ell$ almost surely.
\end{theorem:main}

The fact that the limit $\tfrac{1}{n}\nhat{w_n}$ exists almost surely
with respect to $\P$ follows from a standard application of Kingman's
subadditive ergodic theorem.  The main task of this section is to show
that this limit is strictly greater than zero.  We now state Kingman's
subadditive ergodic theorem \cite{kingman}, using the version from
Woess \cite{woess}*{Theorem 8.10}.

\begin{theorem} \cites{kingman, woess}
\label{theorem:kingman}
Let $(\O,\P)$ be a probability space and $U:\O \to \O$ a measure
preserving transformation. If $W_n$ is a subadditive sequence of
non-negative real-valued random variables on $\O$, that is, $W_{n+k}
\le W_n + W_k \circ U^n$ for all $k,n \in \N$, and $W_1$ has finite
first moment, then there is a $U$-invariant random variable $W_\infty$
such that \[ \lim_{n \to \infty} \tfrac{1}{n}W_n = W_\infty \]
$\P$-almost surely, and in expectation.
\end{theorem}

We will choose $(\O,\P)$ to be the path space $(\gz, \P)$ for the
random walk determined by $(G, \mu)$. We will choose $U$ to be the
Bernoulli shift in the space of increments of the random walk, which
is an ergodic measure-preserving transformation on $(\gz,\P)$.  Set
$W_n = \nhat{w_n}$. Then $W_k \circ U^n = \nhat{w_n^{-1}w_{n+k}}$,
which is the relative distance between $w_n$ and $w_{n+k}$, so the
triangle inequality implies that $W_n$ is subadditive. We have assumed
that $\mu$ has finite first moment with respect to the word metric on
$G$, and as $\nhat{g} \le \norm{g}$ for all $g \in G$, this implies
that $\mu$ also has finite first moment with respect to the relative
metric on $G$.  Theorem \ref{theorem:kingman} then implies that
$\lim_{n \to \infty} \tfrac{1}{n}\nhat{w_n} = \ell$ exists almost
surely, and in expectation, and in fact is constant almost surely. As
$\nhat{w_n} \ge 0$ for all $n$, this implies that $\ell \ge 0$.  It
remains to show that $\ell$ is strictly larger than zero. As the limit
is constant almost surely, and the limit exists in expectation, it
suffices to show that the limit $\tfrac{1}{n}\E(\nhat{w_n})$ is
bounded away from zero.  We will show that for $m$ sufficiently large,
the expected difference between $\nhat{w_n}$ and $\nhat{w_{n+m}}$ is
bounded away from zero.

\begin{lemma} \mylabel{lemma:progress}
There is a constant $N$, such that if $m > N$, then $\E(\nhat{w_{n+m}} -
  \nhat{w_n}) \geqslant \delta > 0$. The constant $N$ depends on $\mu$,
  but is independent of $m$ and $n$.
\end{lemma}

This suffices to prove Theorem \ref{theorem:main}, as we now
explain. Consider
\begin{align*}
\E(\nhat{w_{2N}}) & = \E(\nhat{w_{2N}} - \nhat{w_N} + \nhat{w_N} - \nhat{w_0})
\intertext{This is equal to}
& = \E(\nhat{w_{2N}} - \nhat{w_N}) + \E(\nhat{w_N}) \ge 2\delta
\end{align*}
A similar argument shows $\E(\nhat{w_{kN}}) \ge k\delta$. As $\lim_{n
  \to \infty} \tfrac{1}{n}\nhat{\w_n} = \ell$ is constant almost
surely, this shows $\ell \ge \delta/N > 0$, as required.

We now prove Lemma \ref{lemma:progress}.

\begin{proof}
Consider doing a random walk of length $n$, followed by one of length
$m$. We can compute the expected change in relative length from time
$n$ to time $n+m$, which we shall denote $\D_{n,m}$.
\begin{align*} 
\Delta_{n,m} & = \E(\nhat{w_{n+m}} - \nhat{w_n})
\intertext{We may rewrite this using the definition of expected value,
giving}
\Delta_{n,m} & = \sum_{x \in G}
\mun{n}(x) \sum_{y \in G} \mun{m}(y) (\nhat{xy} - \nhat{x}).
\end{align*} 
The fact that $\mu$ has finite first moment means that this sum is
absolutely convergent, so we can swap the order of summation.
Furthermore, as $\dhat{1,xy} = \dhat{x^{-1},y}$, and $\dhat{1,x} =
\dhat{x^{-1},1}$ we can rewrite this as
\[ \D_{n,m} = \sum_{y \in G} \mun{m}(y) \sum_{x \in G} \mun{n}(x) (
\dhat{x^{-1},y} - \dhat{x^{-1},1} ) \] 
We may split the sum up into two parts, depending on whether or not
$\nhat{y} > A$, where $A$ is a constant which depends on $\mu$. We
will choose $A \ge \Knest + 2\Kbound + \Kexp + 12 \delta$, and
furthermore, we will choose $A$ to be sufficiently large such that
$rL^r < \delta/2Q$, for all real numbers $r \ge A$, and this is
possible as $L < 1$. Here the constants $\Knest, \Kbound, \Kexp, L$ and
$Q$ are those defined previously in Sections \ref{section:halfspaces}
and \ref{section:decay}, though we emphasize that $Q$ and $L$ are the
constants defined for the reflected random walk $(G, \rmu)$. We shall
write $\Bhat_A$ for the ball of radius $A$ about $1$ in $\Grel$.
\begin{align} 
\D_{n,m} = & \sum_{y \in \Bhat_A} \mun{m}(y) \sum_{x \in G} \mun{n}(x)
( \dhat{x^{-1},y} - \dhat{x^{-1},1} ) \label{equation:smalla}  \\ 
+ & \sum_{y \in G \setminus \Bhat_A} \mun{m}(y) \sum_{x \in G} \mun{n}(x) (
\dhat{x^{-1},y} - \dhat{x^{-1},1} ) \label{equation:biga}
\end{align}

We now find lower bounds for lines \eqref{equation:smalla} and
\eqref{equation:biga} in turn. First, in line \eqref{equation:smalla},
$\dhat{x^{-1},y} - \dhat{x^{-1},1} \ge - \nhat{y}$, by the triangle
inequality. We have also assumed that $\nhat{y} \le A$. This implies
the following lower bound for \eqref{equation:smalla}.
\begin{equation}
\eqref{equation:smalla} \ge - \mun{m}(\Bhat_A) A \label{equation:61 bound}
\end{equation}
We now find a lower bound for line \eqref{equation:biga}, and we will
simplify the notation by replacing $x$ with $x^{-1}$. This makes no
difference to the sum as we are summing over all elements of $G$.
\[
\eqref{equation:biga} \ge \sum_{y \in G \setminus \Bhat_A}
\mun{m}(y) \sum_{x \in G} \mun{n}(x^{-1}) ( \dhat{x,y} - \dhat{x,1} )
\]
Let $N$ be the collection of $x \in G$ for which $\dhat{x,y} -
\dhat{x,1} \le 0$, i.e.  $N$ is the halfspace $H(1,y)$. Let $P$ be the
region on which $\dhat{x,y} - \dhat{x,1} \geqslant 3\delta$, so $P$ is
the halfspace $H(y,1;-3\delta)$, using the notation from Proposition
\ref{prop:coarse}. This is illustrated below in Figure \ref{picture3}.
\begin{figure}[H]
\begin{center}
\epsfig{file=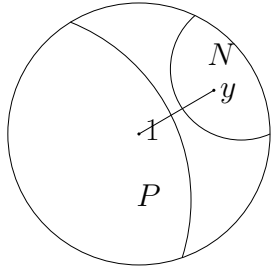, height=100pt}
\end{center}
\caption{Positive and negative regions.} \mylabel{picture3}
\end{figure}
Let $z$ be a point on $[1,y]$ distance $12\delta$ from $y$. We now
show that the region $P$ contains the halfspace $H(z,1)$.  The
complement of $P$ is contained in $H(1,y;3\delta)$, and by Proposition
\ref{prop:coarse}, the complement of $P$ is contained in $H(1,z)$, for
$\nhat{z} = \nhat{y} - 12\delta$. Therefore $P$ contains the halfspace
$H(z,1)$.
By the triangle inequality, $\dhat{x,y} - \dhat{x,1} \ge-\nhat{y}$ on
$N$, and $\dhat{x,y} - \dhat{x,1} \ge 3\delta$ on $P$. This
gives the following lower bound for line \eqref{equation:biga}.
\begin{align*}
\eqref{equation:biga} & \ge  \sum_{y \in G \setminus \Bhat_A}
\mun{m}(y) \left[ - \nhat{y} \rmun{n}(H(1,y)) + 3\delta \rmun{n}(H(z,1)) \right]
\intertext{Here $\rmun{n}$ is the reflected random walk measure, i.e.
$\rmun{n}(x) = \mun{n}(x^{-1})$. Using that fact that
$\rmun{n}(H(z,1)) \ge 1 - \rmun{n}(H(1,z))$, we obtain}
\eqref{equation:biga} & \ge  \sum_{y \in G \setminus \Bhat_A}
\mun{m}(y) \left[ - \nhat{y} \rmun{n}(H(1,y)) + 3\delta (1 -
\rmun{n}(H(1,z)) ) \right].
\intertext{In line \eqref{equation:biga} we have assumed that
  $\nhat{y} \ge 2\Kbound + \Knest + 12\delta$, so the
  halfspaces $H(1,y)$ and $H(1,z)$ both satisfy the hypotheses of
  Lemma \ref{lemma:mu-n-decay}, so we may estimate $\rmun{n}$ in terms
  of $\rnu$.}
\eqref{equation:biga} & \ge  \sum_{y \in G \setminus \Bhat_A}
\mun{m}(y) \left[  3\delta -\nhat{y} Q\rnu(\overline{H(1,y)}) - \delta
Q\rnu(\overline{H(z,1)}) \right]
\intertext{Here $\rnu$ is the harmonic measure corresponding to the
  reflected random walk defined by $(G,\rmu)$, and $Q$ is a constant
  which depends on $\rmu$, but not on $y$ or $z$. In line
  \eqref{equation:biga} we have assumed that $\nhat{y} \ge \Kexp +
  12\delta$, so we may apply Lemma \ref{lemma:exponential}, which says
  that the harmonic measure of halfspaces decays exponentially.}
\eqref{equation:biga} & \ge  \sum_{y \in G \setminus \Bhat_A}
\mun{m}(y) \left[ 3\delta - Q\nhat{y}L^{\nhat{y}} - Q\delta
  L^{\nhat{z}} \right]
\intertext{Now using the facts that $\nhat{y} \geqslant \delta$, and we
  chose $A$ such that $rL^r \le \delta /2Q$ for all $r \ge A$, this
  implies that the term in square brackets in the line above is at least
  $2\delta$. Therefore}
\eqref{equation:biga} & \ge (1 - \mun{m}(\Bhat_A) )
2\delta. \stepcounter{equation} \tag{\theequation} \label{equation:62 bound}
\end{align*}

Recall that $\D_{n,m} = \eqref{equation:smalla} +
\eqref{equation:biga}$, and now using the lower bounds from lines
\eqref{equation:61 bound} and \eqref{equation:62 bound}, we obtain
\[ \D_{n,m} \ge 2\delta - \mun{m}(\Bhat_A) (A + 2\delta). \]
The harmonic measure of a bounded set in $\Grel$ is zero, and so the
random walk is transient on relatively bounded sets. In particular,
for fixed $A$ there is an $N$ such that
$\mun{m}(\Bhat_A) < \delta/(A+2\delta)$
for all $m \ge N$, and this will be our choice of $N$, which depends
on $\mu$. In particular, this means that $\Delta_{n,m} \ge \delta$,
for all $n$, and for all $m \ge N$, and so $\Delta_{n,m}$ is bounded
away from zero, as required.
\end{proof}

This completes the proof of lemma \ref{lemma:progress}, and hence of
Theorem \ref{theorem:main}.



\end{document}